\documentclass[reqno]{amsart}
\usepackage{graphicx}
\usepackage{fullpage}

\usepackage{array,booktabs}

\usepackage{mathrsfs,textcomp,mathtools,accents}
\usepackage{hyperref}
\usepackage{bbm,eufrak}

\usepackage{enumerate}
\usepackage[shortlabels]{enumitem}
\setenumerate{label={\upshape(\roman*)},
  wide=0pt,topsep=0.2cm,labelsep=0em,itemsep=0.1cm,leftmargin=*}

\makeatletter
\@namedef{subjclassname@2020}{%
  \textup{2020} Mathematics Subject Classification}
\makeatother

\newtheorem{thm}{Theorem}[section]

\newtheorem{lem}[thm]{Lemma}
\newtheorem{prop}[thm]{Proposition}
\theoremstyle{definition}
\newtheorem{defn}[thm]{Definition}
\theoremstyle{remark}
\newtheorem{rem}[thm]{Remark}
\theoremstyle{remark}
\newtheorem{exam}[thm]{Example}
\numberwithin{equation}{section}

\newcommand{\set}[1]{\left\{#1\right\}}
\newcommand{\lquot}[2]{#1\text{\textbackslash}\hspace{.1em}#2}
\newcommand{\freccia}{\leftrightarrow}
\newcommand{\R}{\mathbb R}
\newcommand{\HP}{\mathbb H}

\newcommand{\GG}{\mathbb G}

\newcommand{\N}{\mathbb N}
\newcommand{\Z}{\mathbb Z}

\newcommand{\Q}{\mathbb Q}
\newcommand{\WW}{\mathbb W}

\newcommand{\CCC}{{\mathcal C}}

\newcommand{\FFF}{{\mathcal F}}
\newcommand{\GGG}{{\mathcal G}}
\newcommand{\III}{{\mathcal I}}
\newcommand{\JJJ}{{\mathcal J}}

\newcommand{\MMM}{{\mathcal M}}
\newcommand{\OOO}{{\mathcal O}}
\newcommand{\PPP}{{\mathcal P}}

\newcommand{\TTT}{{\mathcal T}}
\newcommand{\WWW}{{\mathcal W}}

\newcommand{\DDD}[1]{\operatorname{def}\text{$\left(#1\right)$}}
\newcommand{\eps}{\varepsilon}
\newcommand{\virg}[1]{``#1"}
\newcommand{\geod}{\mathfrak{g}}
\DeclareMathOperator{\Real}{Re}
\DeclareMathOperator{\per}{per}

\usepackage[dvipsnames]{xcolor}

\usepackage{tikz}
\usetikzlibrary{arrows,calc,positioning,spy,intersections}

\definecolor{farey}{rgb}{0.3,0.3,0.3}

\pgfdeclarelayer{background}
\pgfsetlayers{background,main}

\usepackage[font={scriptsize},labelfont={bf}, labelsep={period}]{caption}
\usepackage{mathdots}
\usepackage[figuresright]{rotating}
\usepackage{pdflscape}
\begin{document}

\title{A Poincar\'e map for the horocycle flow on
  $\lquot{PSL(2,\Z)}{\HP}$ and the Stern-Brocot tree}%
\author{Claudio Bonanno}%
\address{Dipartimento di Matematica, Universit\`a di Pisa, Largo Bruno
  Pontecorvo 5, 56127 Pisa, Italy} \email{claudio.bonanno@unipi.it}

\author{Alessio Del Vigna}%
\address{Dipartimento di Matematica, Universit\`a di Pisa, Largo Bruno
  Pontecorvo 5, 56127 Pisa, Italy} \email{delvigna@mail.dm.unipi.it}

\author{Stefano Isola}%
\address{Scuola di Scienze e Tecnologie, Universit\`a degli Studi di Camerino,
via Madonna delle Carceri, 62032 Camerino, Italy} \email{stefano.isola@unicam.it}

\begin{abstract}
We construct a Poincar\'e map $\PPP_h$ for the positive horocycle flow on the modular surface $\lquot{PSL(2,\Z)}{\HP}$, and begin a systematic study of its dynamical properties. In particular we give a complete characterisation of the periodic orbits of $\PPP_h$, and show that they are equidistributed with respect to the invariant measure of $\PPP_h$ and that they can be organised in a tree by using the Stern-Brocot tree of rational numbers. In addition we introduce a time-reparameterisation of $\PPP_h$ which gives an insight into the dynamics of the non-periodic orbits. This paper constitutes a first step in the study of the dynamical properties of the horocycle flow by purely dynamical methods.
\end{abstract}

\thanks{The authors are partially supported by the research
  project PRIN 2017S35EHN ``Regular and stochastic
  behaviour in dynamical systems'' of the Italian Ministry of
  Education and Research. This research is part of the authors' activity within the
 the UMI Group ``DinAmicI'' \texttt{www.dinamici.org} and the INdAM (Istituto Nazionale di Alta Matematica) group GNFM. The authors are grateful to Giovanni Forni, Andreas Knauf and Jens Marklof for useful discussions and suggestions. }

\maketitle



\section{Introduction} \label{sec:intro}

Hyperbolic geometry has been a crucial ground for the development of important ideas of ergodic theory since its inception. Geodesic flows on surfaces of constant negative curvature were the first non-trivial examples of ergodic flows. In addition, several geometric coding techniques have been developed for such flows, stemming from a 1898 work of J. Hadamard, later developed by M. Morse and G. Hedlund in the 1920s and 30s. Subsequently, interest also grew in the study of positive and negative horocycle flows, which move respectively along the stable and unstable manifolds for the geodesic flow. More concretely, these types of flow move tangent vectors sideways along the horocycle whose center is in the direction of the tangent vector.

The geodesic flow is also one of the first systems for which coding was used to study its dynamical properties. It was E. Artin in the 1920s who noticed that by coding a geodesic on the hyperbolic Poincar\'e half-plane $\HP$ by means of the continued fraction expansions of its end points, it was possible to deduce the existence of periodic and everywhere dense geodesic on the modular surface $\MMM:= \lquot{PSL(2,\Z)}{\HP}$. This observation was later developed in the 80s in papers by R. L. Adler and L. Flatto, and by C. Series (see \cite{adler-flatto,series}). Nowadays it is a classical result that the geodesic flow on the unit tangent bundle of the modular surface $S\MMM$ has a Poincar\'e section for which the associated Poincar\'e map has a factor map on one of the coordinates which is the Gauss map. This connection has been very useful in the study of the rich dynamics of the geodesic flow, and has brought to connections between different aspects of the two systems, such as S. G. Dani's correspondence (\cite{dani}) between Diophantine approximation properties and the behaviour of certain orbits of the geodesic flow on the modular surface, and the D. H. Mayer's correspondence (\cite{mayer}) between the Selberg Zeta function and the dynamical zeta function of the Gauss map.

Many dynamical properties have been studied also for the (positive/negative) horocycle flow on surfaces of constant negative curvature (see e.g. \cite{eins-book}). In particular it is known that the dynamics is less rich than that of the geodesic flow. In the case of the modular surface $\MMM$, the only invariant measures on $S\MMM$ are the hyperbolic volume form and the measures supported on the periodic orbits. Instead in the case of compact surfaces, the horocycle flow is uniquely ergodic and has no periodic orbits. Finer and quantitative ergodic properties are also known: the rigidity of the flow and its rate of mixing have been studied by M. Ratner (\cite{ratner1,ratner2}); in the case of the modular surface, the periodic orbits are known to be asymptotically equidistributed on $S\MMM$ (\cite{sarnak,hejhal,stromberg}); the Birkhoff averages of suitably regular observables have been proved to have a polynomial rate of convergence (\cite{flam-forni,ravotti}). However the beautiful results that we have recalled have been obtained by methods of harmonic analysis and representation theory. A notable exception is given by an equidistribution result for long periodic orbits in \cite{stromberg}, which is proved by ergodic methods but does not contain an estimate for the rate of convergence.

The main objective of this work is to construct a Poincar\'e map for the positive horocycle flow on the unit tangent bundle $S\MMM$, and begin a systematic study of its properties. We believe that this paper represents a first step in the study of the dynamical properties of the horocycle flow without using other structures of the modular surface.
First of all we work in the same framework used to approach the geodesic flow. We define a Poincar\'e map $\PPP_h$ for the positive horocycle flow using a cross-section with respect to which the geodesic flow has a Poincar\'e map whose factor map on the first coordinate is the extended Farey map, the slow version of the Gauss map (see Appendix \ref{app:geodesic}). We obtain a map $\PPP_h$ which is defined on a two-dimensional set $\WW$ homeomorphic to $\R^+\times \R^+$, whose action on the first coordinate is defined in terms of matrices in $SL(2,\Z)$, and preserves an infinite measure $\nu$ which is absolutely continuous with respect to the Lebesgue measure (the same holds for the Poincar\'e map of the geodesic flow). The map $\PPP_h$ is defined in Section \ref{sec:poinc-map}, in which we also give the explicit expression of the first return time function to the cross-section along the horocycle flow, thus constructing a suspension flow isomorphic to the horocycle flow on $S\MMM$. Let us stress that it is a suspension flow on a system preserving an infinite measure, thus its ergodic properties, such as mixing, need to be studied by the results of infinite ergodic theory.

In Section \ref{sec:periodic} we study the periodic points of $\PPP_h$. Without appealing to known results, we give a complete characterisation of the periodic points. We prove that there exists a one-parameter family of periodic points which are characterised by having the first coordinate in $\Q^+$ (Theorem \ref{thm-periodic}), and show that the periodic orbits are asymptotically equidistributed on $\WW$ with respect to the invariant measure $\nu$ (Theorem \ref{equid-perorb}). The behaviour of the periodic points may be studied also by looking at their rational coordinate. Using the Stern-Brocot tree, which is a binary tree containing all the rational numbers in $\Q^+$, we construct an algorithm to obtain all the points in a periodic orbit and to list them in the dynamical order (see Appendix \ref{app:algorithm}). It is immediate to realise that as the period of the orbit diverges, the rational coordinates of the points in the periodic orbit cover all the positive rational numbers.

Lastly, in Section \ref{sec:non-periodic} we study the behaviour of the non-periodic orbits. We show that there exists a time-reparameterization $\TTT_h$ of $\PPP_h$ which has factor map on the first coordinate isomorphic to the so-called backward continued fraction map, whereas the second coordinate of $\TTT_h$ is strictly increasing for non-periodic orbits. This can be interpreted as a proof of the existence of a subsequence of times along which the non-closed horocycles converge to the cusp of $\MMM$.

Let us remark that as far as we know the only other Poincar\'e maps for the horocycle flow on the modular surface which have been constructed are those in \cite{ath-che} (see also \cite{marklof-frob}) and in the recent \cite{marklof}. These maps use a cross-section defined in terms of the identification of $S\HP$ with the space of unimodular lattices. In Section \ref{sec:athreya} we discuss the relations with the construction in \cite{ath-che}.

Before getting into the construction of the Poincar\'e map $\PPP_h$ in Section \ref{sec:poinc-map}, we collect what the reader needs to know about the properties of the Stern-Brocot tree in Section \ref{sec:trees} and about the modular surface $\MMM$ in Section \ref{sec:facts}.

\section{Stern-Brocot trees of positive rational numbers} \label{sec:trees}

In this section we recall the Stern-Brocot tree and its permuted
versions, binary trees which contain all the positive rational numbers
in their reduced form exactly once. \emph{Farey sum} of two fractions, defined to be
\[
    \frac ac \oplus \frac bd := \frac{a+b}{c+d}.
\]
The Farey sum is used to generate all the positive fractions. Let
$1/0$ denote the ``fraction'' of infinity, and let
$F_{-1} = \set{0/1,\,1/0}$. Then for $k\geq 0$ we recursively define
the \emph{Stern-Brocot sets} $F_k$ as the set of fractions containing
$F_{k-1}$ and the Farey sum of two consecutive fractions in
$F_{k-1}$. In each set the fractions are arranged in increasing
order. Recall that the Farey sum of two fractions $a/c$ and $b/d$
satisfies $a/c < a/c \oplus b/d < b/d$. Thus we find
\[
F_0=  \set{ \frac 01,\frac 11, \frac 10},\quad F_1 = \set{\frac 01, \frac 12, \frac 11, \frac 21, \frac 10},\quad F_3 =\set{\frac 01, \frac 13, \frac 12, \frac 23, \frac 11, \frac 32, \frac 21, \frac 31, \frac 10},
\]
and so on. We call \emph{Farey pair} two consecutive fractions in a Stern-Brocot set. In the following we use that a Farey pair $p/q,p'/q'$ with $p/q<p'/q'$ satisfies $p'q-q'p=1$.

\subsection{The Stern-Brocot tree} \label{sec:sb-tree} The
Stern-Brocot sets are a strictly increasing sequence of sets which
eventually contain all positive rational numbers in reduced form. One
can then arrange the fractions in a binary tree, the Stern-Brocot tree
$\TTT$, according to the order they appear in the Stern-Brocot
sets. Figure~\ref{fig:SB} shows the first few levels of the
Stern-Brocot tree. The fractions $0/1$ and $1/0$ are the generators of
all the fractions, hence they constitute the first level of the tree,
$\TTT_{-1}$. Then the level $\TTT_0$ is given by $1/1$ the only
fraction in $F_0\setminus F_{-1}$. Recursively, the $k$-th level
$\TTT_k$ of the tree is given by the fractions in
$F_k\setminus F_{k-1}$. The fractions are arranged in increasing
order, and each is connected to a fraction of the previous level. It
is clear by definition of the Stern-Brocot sets that each fraction
$p/q$ in $F_k$ which is not in $F_{k-1}$ is the Farey sum of two
fractions $a/c,b/d$ in $F_{k-1}$, of which exactly one, say $a/c$, is
in the level $\TTT_{k-1}$ of the Stern-Brocot tree. In this case we
say that $a/c$ is a \emph{parent} of $p/q$, which in turn is the
\emph{daughter} of $a/c$. In addition, all fractions in a level
$\TTT_k$, excluding $0/1$ and $1/0$, have two daughters in
$\TTT_{k+1}$. This induces a natural notion of \emph{sisters
  fractions}, two fractions which have a parent in common.

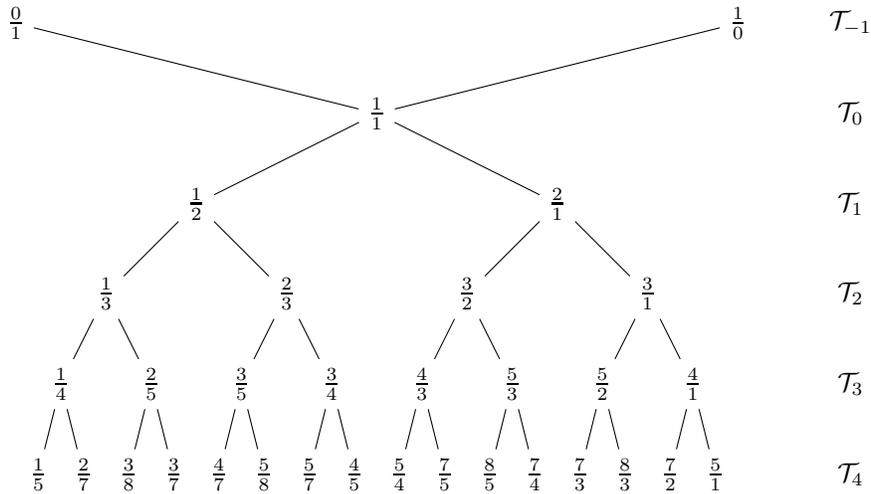
\begin{figure}[ht]
    \begin{tikzpicture}[
    level 1/.style = {sibling distance=4cm},
    level 2/.style = {sibling distance=2cm},
    level 3/.style = {sibling distance=1cm},
    level 4/.style = {sibling distance=0.5cm},
    level distance          = 1cm,
    edge from parent/.style = {draw},
    scale=1.2
    ]

    \node at (-4,1) {$\frac 01$};
    \node at (4,1) {$\frac 10$};
    \draw[thin] (-3.8,0.95) -- (-0.2,0.05);
    \draw[thin] (3.8,0.95) -- (0.2,0.05);

    \foreach \n in {-1,0,1,2,3,4} {
      \pgfmathsetmacro\p{\n}
      \node at (5.25,-\n) {$\mathcal{T}_{\pgfmathprintnumber\p}$};
    }

    \node {$\frac 11$}
    child{
      node {$\frac 12$}
      child{
        node {$\frac 13$}
        child{
          node {$\frac 14$}
          	child{
          	node {$\frac 15$}
          	}
		child{
		node{$\frac 27$}
		}
        }
        child{
          node {$\frac 25$}
          	child{
          	node {$\frac 38$}
          	}
		child{
		node{$\frac 37$}
		}
        }
      }
      child{
        node {$\frac 23$}
        child{
          node {$\frac 35$}
          	child{
          	node {$\frac 47$}
          	}
		child{
		node{$\frac 58$}
		}
        }
        child{
          node {$\frac 34$}
          	child{
          	node {$\frac 57$}
          	}
		child{
		node{$\frac 45$}
		}
        }
      }
    }
    child{
      node {$\frac 21$}
      child{
        node {$\frac 32$}
        child{
          node {$\frac 43$}
          	child{
          	node {$\frac 54$}
          	}
		child{
		node{$\frac 75$}
		}
        }
        child{
          node {$\frac 53$}
          	child{
          	node {$\frac 85$}
          	}
		child{
		node{$\frac 74$}
		}
        }
      }
      child{
        node {$\frac 31$}
        child{
          node {$\frac 52$}
          	child{
          	node {$\frac 73$}
          	}
		child{
		node{$\frac 83$}
		}
        }
        child{
          node {$\frac 41$}
          	child{
          	node {$\frac 72$}
          	}
		child{
		node{$\frac 51$}
		}
        }
      }
    };
\end{tikzpicture}
    \caption{The first six levels of the Stern-Brocot tree.}\label{fig:SB}
\end{figure}

\noindent
The structure of binary tree of $\TTT$ makes it possible to code each
fraction $p/q$ by the path one has to follow on $\TTT$ from the root
$1/1$ to $p/q$ by using $L$ for left and $R$ for right to denote the
motion from a fraction to the left and to the right daughter
respectively. Analogously, using the matrices in $SL(2,\Z)$
\[
  I = \begin{pmatrix} 1& 0 \\ 0 & 1 \end{pmatrix}\, , \quad L
  = \begin{pmatrix} 1& 0 \\ 1 & 1 \end{pmatrix}\, , \quad R
  = \begin{pmatrix} 1& 1 \\ 0 & 1 \end{pmatrix}
\]
one can set
\[
  \frac 11 \freccia I \quad \text{and} \quad \frac pq \freccia
  M\left(\frac pq\right) \in \{L,R\}^*
\]
where a motion on the tree corresponds to right multiplication by $L$
if we move to the left and by $R$ if we move to the right. In this way
\[
  \frac pq \freccia M\left(\frac pq\right) =
  \begin{pmatrix} a& b \\ c & d \end{pmatrix} \in SL(2,\Z) \quad
  \text{and} \quad \frac pq = \frac ac \oplus \frac bd,
\]
so that $p= a+b$ and $q=c+d$.

\begin{exam}\label{esempio-58}
    We have
    \[
        \frac 11 \freccia M\left(\frac 11\right) = I\, , \quad \frac
        12 \freccia M\left(\frac 12\right) = L\, , \quad \frac 21
        \freccia M\left(\frac 21\right) = R\, ,
    \]
    and thus for example
    \[
        \frac 58 \freccia M\left(\frac 58\right) = LRLR
        = \begin{pmatrix} 2 & 3 \\ 3 & 5 \end{pmatrix} \quad
        \text{and} \quad \frac 58 = \frac 23 \oplus \frac 35\, .
    \]
\end{exam}

\noindent
This coding can be extended to all $x\in \R^+$ by using the continued fraction expansion of an irrational number and its slow convergents $\{A_k/B_k\}_k$, so that
\[
\R\setminus \Q \ni x \freccia M(x) \in \{L,R\}^{\N} \quad \text{and} \quad M(x) = \lim_{k\to \infty}\, M\left(\frac{A_k}{B_k}\right).
\]

\begin{rem}\label{rem-LR}
Note that using this coding
\[
x\in [0,1) \Leftrightarrow M(x) = L\cdots \quad \text{and} \quad x\in (1,\infty) \Leftrightarrow M(x) = R\cdots
\]
\end{rem}

\subsection{The permuted Stern-Brocot tree} \label{sec:sb-perm-tree}

\noindent
Starting from $\TTT$ and the $\{L,R\}$ coding of the fractions, the
permuted Stern-Brocot tree $\hat{\TTT}$ is obtained from the
Stern-Brocot tree $\TTT$ by moving each fraction to the position
reached by the motions of its coding on $\TTT$ as read from right to
left. Figure~\ref{fig:SB-perm} shows the first few levels of the the
permuted Stern-Brocot tree. For example, since
\[
\frac 23 \freccia M\left(\frac 23\right) = LR
\]
the fraction $2/3$ can be reached in $\hat{\TTT}$ by the motions $RL$, that is starting from $1/1$ we first go downward to the right and then downward to the left. This sequence of motions describes the new position of $2/3$ in $\hat{\TTT}$.

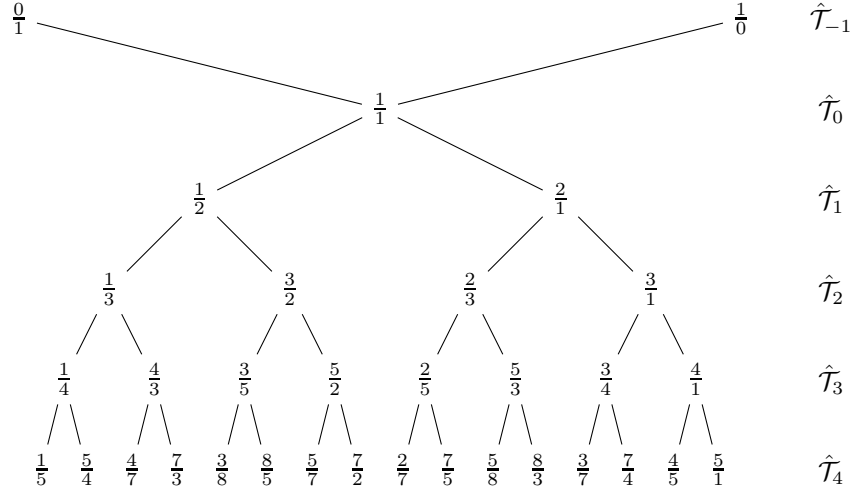
\begin{figure}[ht]
    \begin{tikzpicture}[
    level 1/.style = {sibling distance=4cm},
    level 2/.style = {sibling distance=2cm},
    level 3/.style = {sibling distance=1cm},
    level 4/.style = {sibling distance=0.5cm},
    level distance          = 1cm,
    edge from parent/.style = {draw},
    scale=1.2
    ]

    \node at (-4,1) {$\frac 01$};
    \node at (4,1) {$\frac 10$};
    \draw[thin] (-3.8,0.95) -- (-0.2,0.05);
    \draw[thin] (3.8,0.95) -- (0.2,0.05);

    \foreach \n in {-1,0,1,2,3,4} {
      \pgfmathsetmacro\p{\n}
      \node at (5,-\n) {$\hat{\mathcal{T}}_{\pgfmathprintnumber\p}$};
    }

    \node {$\frac 11$}
    child{
      node {$\frac 12$}
      child{
        node {$\frac 13$}
        child{
          node {$\frac 14$}
          	child{
          	node {$\frac 15$}
          	}
		child{
		node{$\frac 54$}
		}
        }
        child{
          node {$\frac 43$}
          	child{
          	node {$\frac 47$}
          	}
		child{
		node{$\frac 73$}
		}
        }
      }
      child{
        node {$\frac 32$}
        child{
          node {$\frac 35$}
          	child{
          	node {$\frac 38$}
          	}
		child{
		node{$\frac 85$}
		}
        }
        child{
          node {$\frac 52$}
          	child{
          	node {$\frac 57$}
          	}
		child{
		node{$\frac 72$}
		}
        }
      }
    }
    child{
      node {$\frac 21$}
      child{
        node {$\frac 23$}
        child{
          node {$\frac 25$}
          	child{
          	node {$\frac 27$}
          	}
		child{
		node{$\frac 75$}
		}
        }
        child{
          node {$\frac 53$}
          	child{
          	node {$\frac 58$}
          	}
		child{
		node{$\frac 83$}
		}
        }
      }
      child{
        node {$\frac 31$}
        child{
          node {$\frac 34$}
          	child{
          	node {$\frac 37$}
          	}
		child{
		node{$\frac 74$}
		}
        }
        child{
          node {$\frac 41$}
          	child{
          	node {$\frac 45$}
          	}
		child{
		node{$\frac 51$}
		}
        }
      }
    };
\end{tikzpicture}
    \caption{The first six levels of the permuted Stern-Brocot
      tree.}\label{fig:SB-perm}
\end{figure}

\noindent
Given a finite word $w=(w_1 w_2\cdots w_{n-1} w_n)$ let $\hat{w}$ denote its reverse, that is $\hat{w} = (w_n w_{n-1} \cdots w_2 w_1)$. Then $\hat{M}$ denotes the reverse sequence of motions of a rational number, that is $\hat{M}(p/q)$ is the reverse of the word $M(p/q)$, so that looking at the examples we have considered above we have
\[
\hat{M}\left(\frac 23\right) = RL \quad \text{and} \quad \hat{M}\left(\frac 58\right) = RLRL\, .
\]
The following lemma can be proved by using some results in \cite{involution}. For completeness we give a proof in Appendix \ref{app:proof}.

\begin{lem} \label{lem:matr-perm}
The matrices $M(p/q)$ and its reverse $\hat{M}(p/q)$ satisfy the following relation
\[
M\left( \frac pq \right) = \begin{pmatrix} a & b \\ c & d
\end{pmatrix} \quad \Leftrightarrow \quad \hat{M}\left( \frac pq \right) = \begin{pmatrix} d & b \\ c & a
\end{pmatrix}.
\]
\end{lem}

\noindent
It is interesting to notice that $\hat{\TTT}$ is strictly related to
the \emph{extended Farey map} $U:\R^+ \to \R^+$
\begin{equation} \label{ext-farey-map}
    U(x):=
    \begin{cases}
        \dfrac{x}{1-x} & \text{if}\ 0\le x<1\\[0.25cm]
        x-1           & \text{if}\ x \ge 1
    \end{cases}.
\end{equation}
Indeed, given $p/q$ its daughters in $\hat{\TTT}$ are $p/(p+q)$
and $(p+q)/q$ which correspond to the elements of
$U^{-1}(p/q)$. In Appendix \ref{app:geodesic} the map $U$ will pop up as the factor map on the first coordinate of the Poincar\'e map for the geodesic flow with respect to a section $\CCC$ (to be defined below) on the unit tangent $S\MMM$ of the modular surface. For more details about the Stern-Brocot trees and their relations with one-dimensional maps we refer to \cite{orderings}.


\section{Preliminaries from hyperbolic geometry} \label{sec:facts}

We shall consider the \emph{upper half-plane}
\[
    \HP \coloneqq \set{z=x+iy \,:\, x,y\in \R,\ y>0}
\]
with the hyperbolic metric $ds^2 = (dx^2+dy^2)/y^2$, so that
$\HP$ becomes a Riemannian manifold with constant negative curvature
$-1$. The boundary of the hyperbolic plane is
$\partial \HP = \R \cup \{ \infty\}$, where we have set
$\R \coloneqq \{x+i y \,:\, y=0\}$.
Given $\Gamma \in SL(2,\R)$ we let $\Gamma : \HP \to \HP$ act as a M\"obius transformation
\[
\Gamma = \begin{pmatrix}a&b\\c&d\end{pmatrix}\, ,\quad \Gamma(z) = \frac{az+b}{cz+d}.
\]
As the action of $\Gamma$ and $-\Gamma$ coincide, the group of orientation preserving isometries of $\HP$ is taken to be $PSL(2,\R)=SL(2,\R)/\{\pm I\}$.
An important quantity related to a matrix $\Gamma \in SL(2,\R)$ is its \emph{deformation factor} at a point $z$, which we denote by $\DDD{\Gamma(z)}$ and is given by
\begin{equation}\label{def-def}
\Gamma = \begin{pmatrix}a&b\\c&d\end{pmatrix}\, ,\quad \DDD{\Gamma(z)} := \frac{d}{dz} \Gamma(z) = (cz+d)^{-2}
\end{equation}
The deformation factor of a matrix is related to its lift as an action on the tangent bundle of $\HP$.


\noindent
Let us denote by $\zeta \in T_z\HP$ the tangent vectors to $\HP$ at the
point $z$, then to each tangent vector $\zeta$ we associate the angle
$\theta(\zeta) \in (-\pi, \pi]$, which is the angle between $\zeta$ and
the positive $y$-axis measured counterclockwise from the $y$-axis. The
unit tangent bundle $S\HP$ can be identified with
$PSL(2,\R)$, to the effect that for each $(z,\zeta) \in S\HP$
there exists a unique $\Gamma_{(z,\zeta)} \in PSL(2,\R)$ such that
\begin{equation} \label{rep-matrici-sh}
    z=\Gamma_{(z,\zeta)}(i)\quad\text{and}\quad \zeta=d_i(\Gamma_{(z,\zeta)})(i),
\end{equation}
where $i\in T_i\HP$ is the unit vector based at $i\in \HP$ with
$\theta(i)=0$. Note that using the lift $(\Gamma_{(z,\zeta)})_*$ of $\Gamma_{(z,\zeta)}$ to
$S\HP$, we can restate the above conditions as
$(z,\zeta) = (\Gamma_{(z,\zeta)})_* (i,i)$.

\subsection{Geodesic and horocycle flow on $S\HP$}

\noindent
As it is well known, the \emph{geodesics} of $\HP$ are the half lines
orthogonal to $\partial \HP$ and the half circles with center on
$\partial \HP$. Given $(z,\zeta)$ in the unit tangent bundle $S\HP$,
we denote by $\gamma_{z,\zeta}(t)$ the (unique) geodesic tangent to
$\zeta$ at $z$ and set
\[
    \gamma_{z,\zeta}^\pm \coloneqq \lim_{t\to \pm \infty}
    \gamma_{z,\zeta}(t) \in \partial \HP.
\]
Note that $\gamma_{z,\zeta}^+ =\infty$ if and only if
$\theta(\zeta)=0$, and $\gamma_{z,\zeta}^- =\infty$ if and only if
$\theta(\zeta)=\pi$. In these two cases the geodesic is a vertical
half line with the other limit point on $\R$. The \emph{geodesic flow}
is defined as the one-parameter group $g_t$ acting on the unit tangent
bundle $S\HP$ as
\[
    g_t:S\HP \to S\HP, \quad\quad g_t(z,\zeta) =
    (\gamma_{z,\zeta}(t),\gamma'_{z,\zeta}(t)),
\]
where $\gamma'_{z,\zeta}(t)$ is the unit tangent vector to the
geodesic at $ \gamma_{z,\zeta}(t)$. Thanks to the identification of
$S\HP$ with $PSL(2,\R)$, we can read
the geodesic flow on $S\HP$ as the action by right multiplication of
matrices in $PSL(2,\R)$. In particular, for $t\in \R$
\begin{equation} \label{geodesic-sl2r}
   \Gamma_{g_t(z,\zeta)} =\Gamma_{(z,\zeta)} \begin{pmatrix} e^{t/2} &
        0 \\ 0 & e^{-t/2}
    \end{pmatrix}.
\end{equation}
One defines the horocycle flow analogously. Let $W^+(z,\zeta)$
denote the \emph{positive horocycle} at $(z,\zeta)$. If
$\theta(\zeta) \not= 0$, the horocycle $W^+(z,\zeta)$ is the circle
tangent to $\R$ at $\gamma_{z,\zeta}^+$, passing at $z$ and orthogonal
to $\zeta$, with $\zeta$ pointing inward. In this case we denote by
$r^+(z,\zeta)$ its radius. If instead $\theta(\zeta) =0$, then $W^+(z,\zeta)$
is a horizontal line passing at $z$ with $\zeta$ pointing upward. The
\emph{(positive) horocycle flow} is defined as the one-parameter group $h_s^+$ on
$S\HP$ which moves vectors orthogonal to $W^+(z,\zeta)$ rightward on
$W^+(z,\zeta)$ at unit speed. In terms of $PSL(2,\R)$ matrices, for
all $s\in \R$
\begin{equation}\label{horocycle-sl2r}
    \Gamma_{h_s^+(z,\zeta)} = \Gamma_{(z,\zeta)} \begin{pmatrix} 1 & s \\ 0 &
        1
    \end{pmatrix}.
\end{equation}
Thanks to the reinterpretation of the flow action by right
multiplication of matrices, it is straightforward to verify the following
commutation rules:
\[
    g_t\circ g_s = g_s\circ g_t \quad\text{and}\quad g_t\circ h_s^+ =
    h_{se^{-t}}^+\circ g_t \quad , \quad  \forall \, t,s\in \R
\]

\subsection{The modular surface}

\noindent
We now consider the \emph{modular group} $PSL(2,\Z)$, that is the subgroup of $PSL(2,\R)$ generated by the two matrices
\[
    S\coloneqq \begin{pmatrix} 0&1\\-1&0\end{pmatrix}
    \quad\text{and}\quad
    R\coloneqq \begin{pmatrix} 1&1\\0&1\end{pmatrix}.
\]
which act on $\HP$ as
\[
    S(z) = -\frac{1}{z}\quad\text{and}\quad R(z)=z+1
\]
Since
$S^2 = (SR)^3 = I$ the modular group is not free.

\noindent
The \emph{modular surface} is then defined to be the quotient
$\MMM\coloneqq \lquot{PSL(2,\Z)}{\HP}$, with the quotient topology.
The standard fundamental domain $\FFF$ for $\MMM$ is the geodesic
polygon
\[
    \FFF = \set{z=x+iy \in \HP \,:\, |x|\le \frac 12,\ |z|\ge1}.
\]
%
The identification of $S\HP$ with $PSL(2,\R)$
implies that of $S\MMM$ with
$\lquot{PSL(2,\Z)}{PSL(2,\R)}$.
Let $\pi : \HP \to \MMM$ be the projection on the modular surface and
let $\pi_* : S\HP \to S\MMM$ be its lift to the unit tangent
bundles. We let $\tilde g_t :S\MMM \to S\MMM$
and $\tilde h_s^+ :S\MMM \to S\MMM$ be the projection of the geodesic
and horocycle flows on the modular surface, that is
\[
    \tilde g_t = \pi_*\circ g_t\circ \pi_*^{-1}\quad\text{and}\quad
    \tilde h_s^+ = \pi_*\circ h_s\circ \pi_*^{-1}.
\]
Motion on a geodesic or on a horocycle on $\MMM$ corresponds to the
motion on one of the equivalent geodesics or horocycles on $\HP$, up to
identification of equivalent points.


\section{The Poincar\'e map for the horocycle flow and its suspension} \label{sec:poinc-map}

\noindent
Given $n\in \Z$ let $\III_n$  denote the vertical line in $\HP$ given by
\[
  \III_n = \set{x+i y \in \HP\, :\, x=n}
\]
and let
\[
  \III^+ := \III_0 \cap \FFF = \set{z \in \HP\, :\, x=0,\ y\ge 1}
  \quad \text{and} \quad \III^- := \III_0 \setminus \III^+ = \set{z
    \in \HP\, :\, x=0,\ 0<y< 1}
\]
Our Poincar\'e section for the geodesic and the horocycle flows on $S\MMM$ will be the set
\[
    \CCC := \set{ (z,\zeta) \in S\MMM\, :\, z\in \III^+,\
      \theta(\zeta)\not= 0,\pi}
\]
Set moreover $C:=  \set{(z,\zeta) \in S\HP\, :\, \pi_*(z,\zeta) \in \CCC}$.
Since $S(i)=i$ and $S(\III^+\setminus \{ i \}) = \III^-$, whereas its lift $S_* : S\HP \to S\HP$ satisfies
\[
  \theta(S_*(z,\zeta)) = \theta(\zeta)+\pi\, \, \text{(mod
    $2\pi$)}\qquad \forall\, (z,\zeta) \in S\HP \, \, \text{with
    $z\in \III_0$,}
\]
for each $(z,\zeta) \in \CCC$ there exists $(z',\zeta') \in S\HP$ with $z'\in \III_0$ and $\theta(\zeta')\in (-\pi,0)$ such that $\pi_*(z',\zeta') = (z,\zeta)$. Hence in the following we identify $\CCC$ with the set of points on $\III_0$ with tangent unit vector $\zeta$ pointing towards the half-plane of points with strictly positive real part, in accordance with
\begin{equation}\label{not-poin-sec}
\CCC = \pi_* \Big(\set{(z,\zeta) \in S\HP\, :\, z\in \III_0\, ,\, \theta(\zeta) \in (-\pi,0)}\Big)\, ,
\end{equation}
Stated otherwise, when we write $(z,\zeta)\in \CCC$ we think of it as a couple with $z\in \III_0$ and $\theta(\zeta) \in (-\pi,0)$. Analogously we set
\begin{equation}\label{equiv-poin-sec}
C = \set{(z,\zeta) \in S\HP\, :\, \exists \, \Gamma \in PSL(2,\Z) \, \text{ such that }\, \Gamma(z)\in \III_0\, ,\, \theta(\Gamma_*(z,\zeta))\in (-\pi,0)}\, .
\end{equation}
Subsets of $C$ which play an important role in the rest of the paper are points with base on the lines $\III_n$ and on the half-circles
\[
    \JJJ_n := \set{x+i y\in \HP\, :\,
      \left(x-n-\frac{1}{2}\right)^2 + y^2 = \frac 14}\, .
\]
Note that
$\III_n = R^n(\III_0)$ and $\JJJ_n = R^n L(\III_0)$, where
\[
L\coloneqq \begin{pmatrix} 1&0\\1&1\end{pmatrix} \in PSL(2,\Z)
\]
acts on $\HP$ by $L(z) = z/(z+1)$.

\noindent
Note moreover that for each $n\in \Z$ the lines $\III_n$,
$\JJJ_n$ and $\III_{n+1}$ are the sides of the hyperbolic triangle
\[
    \Delta_n := \set{x+i y \in \HP\, :\, n\le x\le n+1,\
      \left(x-n-\frac{1}{2}\right)^2 + y^2 \ge \frac 14}
\]
with vertices in $n,\,n+1 \in \R$ and $\infty$ (each $\Delta_n$ contains three copies of the fundamental domain $\FFF$).

\vskip 0.1cm
\noindent
We now proceed with the construction of the Poincar\'e map on $\CCC$ for the horocycle flow
$\tilde h^+_s$ for negative times $s$ (the Poincar\'e map on $\CCC$ for the geodesic flow $ \tilde g_t$ is already known, but we remind its construction in Appendix
\ref{app:geodesic}).
Let's start by defining the set
\[
    \WW :=\set{(\gamma,r,\eps)\in \R^+ \times \R^+ \times
      \{-1,0,+1\}\, :\, r\ge \gamma,\ \text{$\eps=0$ if and only if
        $r=\gamma$}}
\]
which is homeomorphic to $\R^+ \times \R^+$. Now, for each
$(z,\zeta)\in \CCC$ the positive horocycle $W^+(z,\zeta)$ is tangent
to $\R$ at $\gamma^+_{z,\zeta}$ and intersects $\III_0$. Its
radius thus satisfies $r^+(z,\zeta)\ge \gamma^+_{z,\zeta}$, with strict inequality if and only if
$W^+(z,\zeta)$ intersects $\III_0$ in two different points. We can therefore code
$(z,\zeta)\in \CCC$ using its horocycle $W^+(z,\zeta)$ by a point in
$\WW$ by setting
\[
    \gamma:= \gamma_{z,\zeta}^+,\quad
    r:=r^+(z,\zeta),\quad\text{and}\quad \eps =
    \begin{cases}
        -1 & \text{if $\theta(\zeta) \in (-\pi/2, 0)$}\\
        0  & \text{if $\theta(\zeta) = -\pi/2$}\\
        $+1$ & \text{if $\theta(\zeta) \in (-\pi, -\pi/2)$}
    \end{cases}.
\]
The variable $\eps$ is equal to 0 if the horocycle $W^+(z,\zeta)$ is
tangent to $\III_0$ at $z$, whereas is equal to $1$ or to $-1$ if $z$ is
the point of intersection of $W^+(z,\zeta)$ with $\III_0$ of highest,
respectively lowest, variable $y$. We have thus defined a map
\begin{equation} \label{eq-map-W} \WWW : \CCC \to \WW\, ,\quad
    (z,\zeta) \mapsto \WWW(z,\zeta) = (\gamma_{z,\zeta}^+,\,
    r^+(z,\zeta),\, \eps).
\end{equation}
Given $(z,\zeta)\in \CCC$ let us consider the horocycle flow $\tilde h^+_s(z,\zeta)$ for $s\in (-\infty,0)$, and let $(z',\zeta')$ be its first return to $\CCC$. We can thereby view the Poincar\'e map of $\tilde h_s^+$ on $\CCC$ as a map
\[
\PPP_h : \WW \to \WW
\]
\[
    (\gamma,r,\eps) = \WWW(z,\zeta) \mapsto \PPP_h(\gamma,r,\eps) =
    (\gamma',r',\eps')= \WWW(z',\zeta')\, .
\]
It is clear that $(z',\zeta')$, viewed as a point of the set $C$ (with slight abuse of notation), depends on
the configuration of $W^+(z,\zeta)$ and on the value of $\eps$ (see
Figures \ref{fig-case+1} and \ref{fig-case-1} below). In particular
the first return to $C$ for $\tilde h_s^+$ with negative $s$ occurs on
one of the following sets: $\JJJ_{-1}$, $\III_{-1}$ or $\III_0$ if
$\eps=-1$; $\III_1$, $\III_2$ or $\JJJ_0$ if $\eps=0,+1$ (here we are
using that $\theta(\zeta')\not= -\pi$).
Then $W^+(z',\zeta')$ will be the
positive horocycle defined by $(z',\zeta')$ as a point in $\CCC$. If
$\Gamma\in PSL(2,\Z)$ sends $z'\in C$ to
$\III_0$, that is $\Gamma(z')\in \III_0$ is the representative of $z'$
in $\CCC$, then $W^+(z',\zeta')=\Gamma(W^+(z,\zeta))$.

\noindent
Our aim is now to describe the map $\PPP_h$ on the points of $\WW$,
also in relation with the $\{L,R\}$ coding of $\gamma \in \R^+$
introduced in Section \ref{sec:trees}. The following lemma is a key result in order to characterize the map $\PPP_h$ in dependence on the subset of $C $ on which the first return occurs.

\begin{lem}\label{lem:via-matrici}
    For $(z,\zeta)\in \CCC$, let $(\gamma,r,\eps) = \WWW(z,\zeta)$. If
    the first return to $C$ for $\tilde h_s^+$ with negative $s$ along
    $W^+(z,\zeta)$ occurs on $\Gamma^{-1}(\III_0)$ for some
    $\Gamma\in PSL(2,\Z)$, then
    $(\gamma',r',\eps')=\PPP_h(\gamma,r,\eps)$ satisfies
    \[
        \gamma' = \Gamma(\gamma)\quad \text{and} \quad r' =
        \DDD{\Gamma(\gamma)}\, r\, ,
    \]
    where $\DDD{\Gamma(\cdot)}$ is the deformation factor defined in
    \eqref{def-def}.
\end{lem}

\begin{proof}
    The statement for $\gamma'$ simply follows from
    $W^+(z',\zeta')=\Gamma(W^+(z,\zeta))$ whenever the first return occurs
    on $\Gamma^{-1}(\III_0)$. Let us now see how the radius of
    the positive horocycle $W^+(z,\zeta)$ changes when when a matrix in $PSL(2,\Z)$ acts on it.  We first
    consider the action of the generators $R$ and $S$. If
    $\Gamma\in \set{R,R^{-1}}$ then the horocycle $\Gamma(W^+(z,\zeta))$ is
    obtained by just translating $W^+(z,\zeta)$ to the right or to the
    left. Therefore the radius does not change, and
    $\DDD{\Gamma(z)} = 1$ for all $z$. If, on the other hand, $\Gamma = S=S^{-1}$,
    one easily checks that
    $x+iy \in W^+(z,\zeta)$, a circle of radius $r$ and center at
    $\gamma + ir$, if and only if $(x-\gamma)^2+y^2-2ry=0$, and
    $S(x+iy) = X+iY$ satisfies $(X-\gamma')+Y^2-2r' Y=0$ with
    $\gamma'=-1/\gamma$ and $r'=r/\gamma^2$. Since
    $\DDD{S(\gamma)}= \gamma^{-2}$ the result is proven also in this
    case. The general case now follows by noting that
\[
\DDD{\Gamma_1\Gamma_2(z)}=\DDD{\Gamma_1(\Gamma_2(z))}\DDD{\Gamma_2(z)}
\]
for all $z$ and all $\Gamma_1,\Gamma_2\in PSL(2,\Z)$.
\end{proof}

The following tables describe the different cases that have to
be considered to define $\PPP_h$ on $\WW$, providing the following
information: the subset of $C$ on which the first return occurs; the
matrix $\Gamma$ as in Lemma \ref{lem:via-matrici}; the image
$(\gamma',r',\eps')= \PPP_h(\gamma,r,\eps)$ computed using the lemma
along with simple arguments for $\eps'$; the word $M(\gamma')\in \{L,R\}^*$
as a function of $M(\gamma)$ as defined in Section \ref{sec:trees},
using the notation $\sigma: \{L,R\}^* \to \{L,R\}^*$ for the shift
map.

\noindent
We start with the case $\eps=+1$, for which we refer to Figure
\ref{fig-case+1}. We need to distinguish between the cases $\gamma>1$
and $\gamma<1$. In the first case the first return occurs on $\III_1$,
but in the last case it may occur on $\III_1$ or $\JJJ_0$ depending on
the value of the radius $r$ of the horocycle. The case $\gamma=1$ is
special because the case $\theta(\zeta')=0$ is not admissible for
points in $\CCC$. The different possibilities are described in Table \ref{tab-eps+1}. We
remark that in some circumstances it is necessary to add the action of $S$ to
have $\theta(\zeta')$ in the interval $(-\pi,0)$. Moreover the action
on $M(\gamma)$ has been computed by concatenation of the letters
describing $\Gamma$ and $M(\gamma)$, using some relations among the
matrices such as $SR^{-1} = RL^{-1}$.

\begin{table}[h]
    \begin{tabular}{c|c|c|c|c}
        \textbf{Conditions on $\gamma$ and $r$} & \textbf{First return} & $\Gamma$ & $\PPP_h(\gamma,r,\eps)$ & $M(\gamma')$\\
        \specialrule{1pt}{1pt}{1pt}
        $\gamma >1$ & $\III_1$ & $R^{-1}$ & $\left(\gamma -1, r, +1\right)$ & $\sigma(M(\gamma))$ \\
        \hline
        $\gamma=1$ & $\III_2$ & $SR^{-2}$ & $\left(1, r, -1\right)$ & $M(\gamma)$ \\
        \hline
        $\gamma<1$, $r> 1-\gamma$ & $\III_1$ & $SR^{-1}$ & $\left(\frac{1}{1-\gamma}, \frac{r}{(1-\gamma)^2}, -1\right)$ & $R\sigma(M(\gamma))$ \\
        \hline
        $\gamma<1$, $r= 1-\gamma$ & $\III_1$ & $SR^{-1}$ & $\left(\frac{1}{1-\gamma}, \frac{1}{1-\gamma}, 0\right)$ & $R\sigma(M(\gamma))$ \\
        \hline
        $\gamma<1$, $r< 1-\gamma$ & $\JJJ_0$ & $L^{-1}$ & $\left(\frac{\gamma}{1-\gamma}, \frac{r}{(1-\gamma)^2}, +1\right)$ & $\sigma(M(\gamma))$ \\
    \end{tabular}
    \caption{Definition of the map $\PPP_h$ for
      $\eps=+1$.}\label{tab-eps+1}
\end{table}

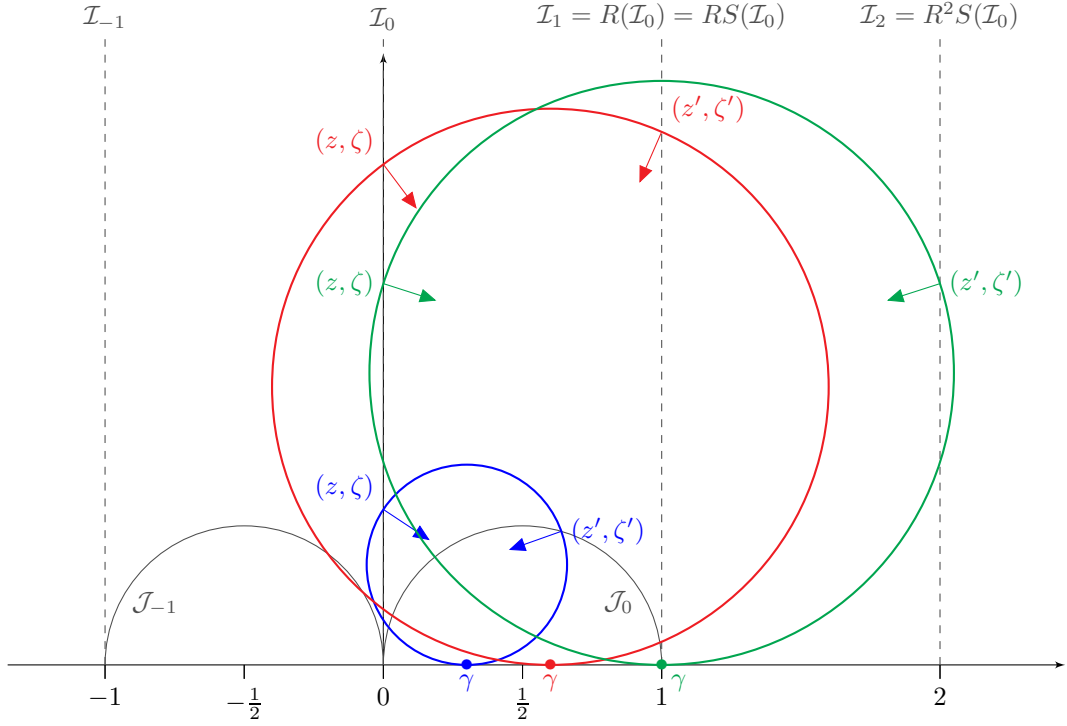
\begin{figure}[h!]
    \pgfmathsetmacro{\myxlow}{-1.25}
\pgfmathsetmacro{\myxhigh}{2.25}
\pgfmathsetmacro{\myiterations}{2}

\begin{tikzpicture}[scale=3.7]
    \draw[-latex',thin](\myxlow-0.1,0) -- (\myxhigh+0.2,0);
    \draw[-latex',thin,name path=I0](0,0) -- (0,2.2);

    \pgfmathsetmacro{\succofmyxlow}{\myxlow+1}
    \foreach \x in {-1,-0.5,0,0.5,1,2}
    {
      \draw (\x,0) -- (\x,-0.05);
    }
    \foreach \x in {-1,0,1,2}
    {
      \draw (\x,0) -- (\x,-0.05) node[below]{$\x$};
    }
    \draw (-0.5,0) -- (-0.5,-0.05) node[below]{$-\frac 12$};
    \draw (0.5,0) -- (0.5,-0.05) node[below]{$\frac 12$};

    \draw[dashed,farey] (-1,0) -- (-1,\myxhigh) node[above]{$\III_{-1}$};
    \draw[dashed,farey] (0,0) -- (0,\myxhigh) node[above]{$\III_0$};
    \draw[dashed,farey,name path=I1] (1,0) -- (1,\myxhigh) node[above]{$\III_1 = R(\III_0)=RS(\III_0)$};
    \draw[dashed,farey,name path=I2] (2,0) -- (2,\myxhigh) node[above]{$\III_2=R^2S(\III_0)$};
    \draw[thin,farey,name path=J-1] (0,0) arc(0:180:0.5);
    \draw[thin,farey,name path=J0] (1,0) arc(0:180:0.5);
    \node [farey] at (165:0.85) {$\JJJ_{-1}$};
    \node [farey] at (15:0.875) {$\JJJ_0$};

    \def\g{0.3}
    \def\r{0.36}
    \draw[blue,thick,name path=Hblu] (\g,\r) circle (\r);
    \node[blue] at (\g,0) {$\bullet$};
    \node[blue,below] at (\g,0) {$\gamma$};
    \path [name intersections={of=I0 and Hblu,by=z}];
    \node [blue,above left] at (z) {$(z,\zeta)$};
    \path [name intersections={of=J0 and Hblu,by=z'}];
    \node [blue,right] at (z') {$(z',\zeta')$};
    \draw[blue,-triangle 45] (z) -- ($(z)!0.2cm!(\g,\r)$);
    \draw[blue,-triangle 45] (z') -- ($(z')!0.2cm!(\g,\r)$);

    \def\g{0.6}
    \def\r{1}
    \draw[Red,thick,name path=Hblu] (\g,\r) circle (\r);
    \node[Red] at (\g,0) {$\bullet$};
    \node[Red,below] at (\g,0) {$\gamma$};
    \path [name intersections={of=I0 and Hblu,by=z}];
    \node [Red,above left] at (z) {$(z,\zeta)$};
    \path [name intersections={of=I1 and Hblu,by=z'}];
    \node [Red,above right] at (z') {$(z',\zeta')$};
    \draw[Red,-triangle 45] (z) -- ($(z)!0.2cm!(\g,\r)$);
    \draw[Red,-triangle 45] (z') -- ($(z')!0.2cm!(\g,\r)$);

    \def\g{1}
    \def\r{1.05}
    \draw[Green,thick,name path=Hblu] (\g,\r) circle (\r);
    \node[Green] at (\g,0) {$\bullet$};
    \node[Green,below right] at (\g,0) {$\gamma$};
    \path [name intersections={of=I0 and Hblu,by=z}];
    \node [Green,left] at (z) {$(z,\zeta)$};
    \path [name intersections={of=I2 and Hblu,by=z'}];
    \node [Green,right] at (z') {$(z',\zeta')$};
    \draw[Green,-triangle 45] (z) -- ($(z)!0.2cm!(\g,\r)$);
    \draw[Green,-triangle 45] (z') -- ($(z')!0.2cm!(\g,\r)$);

\end{tikzpicture}
    \caption{Action of the first return map $\PPP_h$ for
      $\eps=+1$. The red horocycle has $\gamma<1$ and $r>1-\gamma$, so
      that the first return occurs on $\III_1$; the blue horocycle has
      $\gamma<1$ and $r<1-\gamma$, so that the first return occurs on
      $\JJJ_0$; the green horocycle has $\gamma=1$ so that the first
      return occurs on $\III_2$.}\label{fig-case+1}
\end{figure}

\noindent
In the case $\eps=-1$, following $\tilde h^+_s$ with $s$ negative, the
points slide on the negative abscissa half-plane, so the first return
may occur on $\JJJ_{-1}$ or on $\III_{-1}$, or back on $\III_0$ (see
Figure~\ref{fig-case-1}). The different possibilities are described in
Table~\ref{tab-eps-1}. In this case note that
$r\in (1+\gamma,\gamma(1+\gamma))$ is possible only when $\gamma> 1$,
and it corresponds to the only situation for which the return occurs
on $\III_{-1}$, since the horocycle does not intersect $\JJJ_{-1}$ but
does intersect $\III_{-1}$.

\begin{table}[h]
    \begin{tabular}{c|c|c|c|c}
        \textbf{Conditions on $\gamma$ and $r$} & \textbf{First return} & $\Gamma$ & $\PPP_h(\gamma,r,\eps)$ & $M(\gamma')$\\
        \specialrule{1pt}{1pt}{1pt}
        $r>\gamma(1+\gamma)$ & $\JJJ_{-1}$ & $L$ & $\left(\frac{\gamma}{1+\gamma}, \frac{r}{(1+\gamma)^2}, -1\right)$ & $L\, M(\gamma)$ \\
        \hline
        $r=\gamma(1+\gamma)$ & $\JJJ_{-1}$ & $L$ & $\left(\frac{\gamma}{1+\gamma}, \frac{\gamma}{1+\gamma}, 0 \right)$ & $L\, M(\gamma)$ \\
        \hline
        $\gamma > 1$, $1+\gamma<r<\gamma(1+\gamma)$ & $\III_{-1}$ & $R$ & $\left(\gamma+1, r, -1\right)$ & $R\, M(\gamma)$ \\
        \hline
        $\gamma > 1$, $r=1+\gamma$ & $\III_{-1}$ & $R$ & $\left(\gamma+1, \gamma+1, 0\right)$ & $R\, M(\gamma)$ \\
        \hline
        $r<\min\{1+\gamma, \gamma(1+\gamma)\}$ & $\III_{0}$ & $I$ & $\left(\gamma, r, +1\right)$ & $M(\gamma)$ \\
    \end{tabular}
    \caption{Definition of the map $\PPP_h$ for $\eps=-1$.}
    \label{tab-eps-1}
\end{table}

\begin{figure}[h!]
    \pgfmathsetmacro{\myxlow}{-1.35}
\pgfmathsetmacro{\myxhigh}{2.25}
\pgfmathsetmacro{\myiterations}{2}

\begin{tikzpicture}[scale=3.7]
    \clip (\myxlow,-0.5) rectangle (\myxhigh+0.2,\myxhigh+0.2);

    \draw[-latex',thin](\myxlow-0.1,0) -- (\myxhigh+0.2,0);
    \draw[-latex',thin,name path=I0](0,0) -- (0,2.2);

    \pgfmathsetmacro{\succofmyxlow}{\myxlow+1}
    \foreach \x in {-1,-0.5,0,0.5,1,2}
    {
      \draw (\x,0) -- (\x,-0.05);
    }
    \foreach \x in {-1,0,1,2}
    {
      \draw (\x,0) -- (\x,-0.05) node[below]{$\x$};
    }
    \draw (-0.5,0) -- (-0.5,-0.05) node[below]{$-\frac 12$};
    \draw (0.5,0) -- (0.5,-0.05) node[below]{$\frac 12$};

    \draw[dashed,farey,name path=I-1] (-1,0) -- (-1,\myxhigh) node[above]{$\III_{-1}$};
    \draw[dashed,farey] (0,0) -- (0,\myxhigh) node[above,yshift=0.06cm]{$\III_0$};
    \draw[dashed,farey,name path=I1] (1,0) -- (1,\myxhigh) node[above]{$\III_1 = R(\III_0)=RS(\III_0)$};
    \draw[dashed,farey,name path=I2] (2,0) -- (2,\myxhigh) node[above]{$\III_2=R^2S(\III_0)$};
    \draw[thin,farey,name path=J-1] (0,0) arc(0:180:0.5);
    \draw[thin,farey,name path=J0] (1,0) arc(0:180:0.5);
    \node [farey] at (160:0.8) {$\JJJ_{-1}$};
    \node [farey] at (30:1) {$\JJJ_0$};

    \def\g{1.9}
    \def\r{3.2}
    \draw[blue,thick,name path=Hblu] (\g,\r) circle (\r);
    \node[blue] at (\g,0) {$\bullet$};
    \node[blue,below] at (\g,0) {$\gamma$};
    \path [name intersections={of=I0 and Hblu,by=z}];
    \node [blue,below left] at (z) {$(z,\zeta)$};
    \path [name intersections={of=I-1 and Hblu,by=z'}];
    \node [blue,below left] at (z') {$(z',\zeta')$};
    \draw[blue,-triangle 45] (z) -- ($(z)!0.2cm!(\g,\r)$);
    \draw[blue,-triangle 45] (z') -- ($(z')!0.2cm!(\g,\r)$);

    \def\g{0.9}
    \def\r{1.8}
    \draw[Red,thick,name path=Hblu] (\g,\r) circle (\r);
    \node[Red] at (\g,0) {$\bullet$};
    \node[Red,below] at (\g,0) {$\gamma$};
    \path [name intersections={of=I0 and Hblu,by=z}];
    \node [Red,below] at (z) {$(z,\zeta)$};
    \path [name intersections={of=J-1 and Hblu,by=z'}];
    \node [Red,left] at (z') {$(z',\zeta')$};
    \draw[Red,-triangle 45] (z) -- ($(z)!0.2cm!(\g,\r)$);
    \draw[Red,-triangle 45] (z') -- ($(z')!0.2cm!(\g,\r)$);

    \def\g{1.4}
    \def\r{1.45}
    \draw[Green,thick,name path=Hblu] (\g,\r) circle (\r);
    \node[Green] at (\g,0) {$\bullet$};
    \node[Green,below] at (\g,0) {$\gamma$};
    \path [name intersections={of=I0 and Hblu,by={z',z}}];
    \node [Green,left] at (z) {$(z,\zeta)$};
    \node [Green,left] at (z') {$(z',\zeta')$};
    \draw[Green,-triangle 45] (z) -- ($(z)!0.2cm!(\g,\r)$);
    \draw[Green,-triangle 45] (z') -- ($(z')!0.2cm!(\g,\r)$);

\end{tikzpicture}
    \caption{Action of the first return map $\PPP_h$ for
      $\eps=-1$. The red horocycle has $\gamma<1$ and
      $r>\gamma(1+\gamma)$, so that the first return occurs on
      $\JJJ_{-1}$; the blue horocycle has $\gamma>1$ and
      $1+\gamma<r<\gamma(1+\gamma)$, so that it does not intersect
      $\JJJ_{-1}$ and thus the first return occurs on $\III_{-1}$; the
      green horocycle has $\gamma<1$ and
      $r<\min\{1+\gamma, \gamma(1+\gamma)\}$, so that the first return
      occurs on $\III_0$.}\label{fig-case-1}
\end{figure}
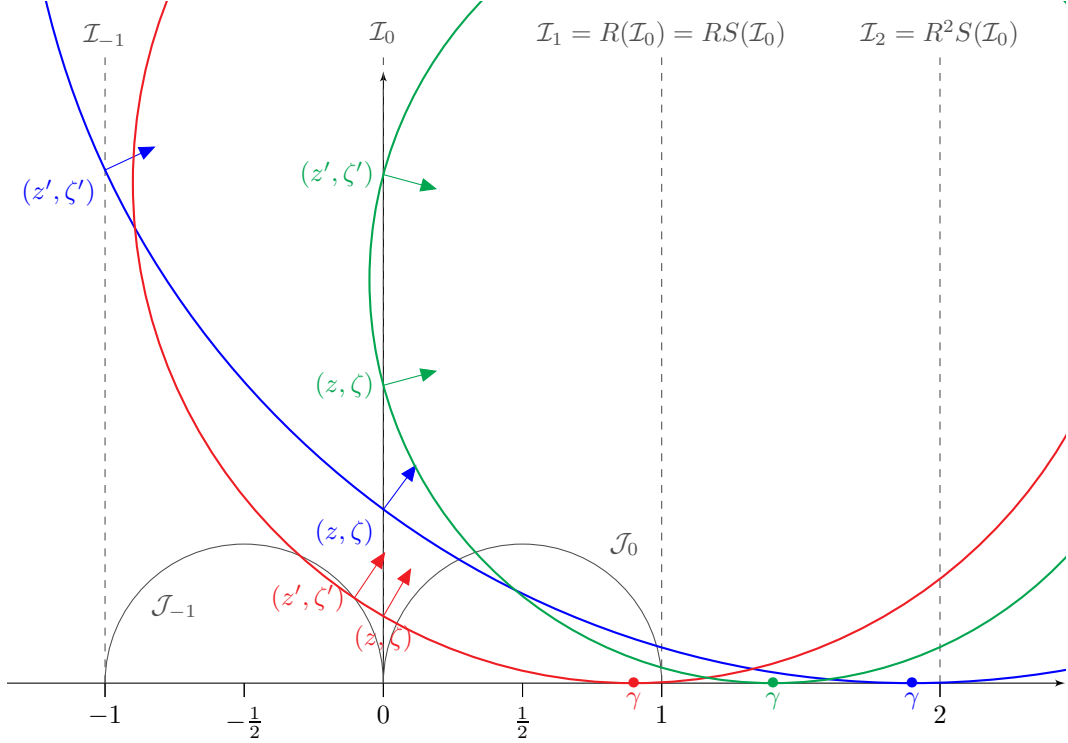

\noindent
It remains to describe the case $\eps=0$. In this case we necessarily have
$r=\gamma$ and everything works analogously to the case
$\eps=+1$. See Table \ref{tab-eps0}.

\begin{table}[h]
    \begin{tabular}{c|c|c|c|c}
        \textbf{Conditions on $\gamma$ and $r$} & \textbf{First return} & $\Gamma$ & $\PPP_h(\gamma,r,\eps)$ & $M(\gamma')$\\
        \specialrule{1pt}{1pt}{1pt}
        $\gamma =r >1$ & $\III_1$ & $R^{-1}$ & $\left(\gamma -1, \gamma, +1\right)$ & $\sigma(M(\gamma))$ \\
        \hline
        $\gamma=r=1$ & $\III_2$ & $SR^{-2}$ & $\left(1, 1, 0\right)$ & $M(\gamma)$ \\
        \hline
        $\frac 12 <\gamma=r<1$ & $\III_1$ & $SR^{-1}$ & $\left(\frac{1}{1-\gamma}, \frac{\gamma}{(1-\gamma)^2}, -1\right)$ & $R\sigma(M(\gamma))$ \\
        \hline
        $\gamma=r = \frac 12$ & $\III_1$ & $SR^{-1}$ & $\left(2, 2, 0\right)$ & $R\sigma(M(\gamma))$ \\
        \hline
        $\gamma=r<\frac 12$ & $\JJJ_0$ & $L^{-1}$ & $\left(\frac{\gamma}{1-\gamma}, \frac{\gamma}{(1-\gamma)^2}, +1\right)$ & $\sigma(M(\gamma))$ \\
    \end{tabular}
    \caption{Definition of the map $\PPP_h$ for $\eps=0$.}
    \label{tab-eps0}
\end{table}

\noindent
Having constructed the map $(\WW,\PPP_h)$, we now proceed to construct the associated suspension flow. Let $(z,\zeta)\in \CCC$ and $(\gamma,r,\eps) := \WWW(z,\zeta)$. The first return of the horocycle flow $\tilde h_s^+$ along $W^+(z,\zeta)$ for negative times $s$ occurs at a time $s_h(z,\zeta)<0$, which can accordingly be written as a function $s_h(\gamma,r,\eps)$. Set
\[
\Sigma_h := \set{(\gamma,r,\eps,\xi) \in \WW \times \R\, :\, s_h(\gamma,r,\eps)\le \xi \le 0}
\]
and use the map $\PPP_h$ to define the flow $\phi_s : \Sigma_h \to \Sigma_h$ by
\begin{equation}\label{susp-flow-horo}
    \phi_s (\gamma,r,\eps,\xi) =
    \begin{cases}
        (\gamma, r, \eps, \xi+s) & \text{if
          $s_h(\gamma,r,\eps)<\xi+s\le 0$}
        \\
        (\PPP_h(\gamma, r,\eps), 0) & \text{if
          $\xi+s=s_h(\gamma,r,\eps)$}
    \end{cases}.
\end{equation}
We made explicit the flow for $s$ negative, the case $s$ positive is defined analogously as $\PPP_h$ is an invertible map.
The next Proposition summarizes some basic ergodic properties of the dynamical systems $(\WW,\PPP_h)$ and
$(\Sigma_h,  \phi_s)$.

\begin{prop}\label{mis-inv-horo}

    \begin{enumerate}[label={\upshape(\roman*)},wide = 0pt,leftmargin=*]
      \item The map $\PPP_h :\WW \to \WW$ preserves the infinite
        measure $\nu$ which is absolutely continuous with respect to
        the Lebesgue measure on $\WW$ with density
        $k(\gamma,r,\eps)=r^{-2}$.
      \item The flow $\phi_s : \Sigma_h \to \Sigma_h$ preserves the
        measure $\tilde \nu$ which is absolutely continuous with
        respect to the Lebesgue measure on $\Sigma_h$ with density
        $\tilde k(\gamma,r,\eps,\xi)=r^{-2}$. Moreover
        $\tilde \nu(\Sigma_h) = 4\pi^2/3$.
      \item The dynamical system $(\Sigma_h, \tilde \nu, \phi_s)$ is
        isomorphic to the system $(S\MMM,\tilde m, \tilde h_s^+)$,
        where $\tilde m$ is the projection on $S\MMM$ of the Liouville
        measure $dm(x,y,\theta) = y^{-2}\, dx\, dy\, d\theta$ on
        $S\HP$, with $z=x+i y$ and $\theta = \theta(\zeta)$.
    \end{enumerate}
\end{prop}

\begin{proof}
    Given $(z,\zeta)\in S\HP$ with $\theta(\zeta)\not\in \set{0,\pi}$, consider its coordinates $(\gamma,r,\xi)$ defined as $\gamma = \gamma_{z,\zeta}^+$, $r=r^+(z,\zeta)$ being the radius of the horocycle $W^+(z,\zeta)$, and $\xi\in \R$ such that $(z,\zeta) = h_\xi^+(z',\zeta')$ where $z'$ is the point in $W^+(z,\zeta)$ with maximum value of the $y$ component. We obtain the following expressions for the coordinates $(x,y,\theta)$ of $(z,\zeta)$, where $z=x+i y$ and $\theta=\theta(\zeta)$:
    \begin{equation}\label{coord-cr-xytheta}
        x(\gamma,r,\xi) = \gamma - \frac{2r \xi}{\xi^2+1}\, ,\quad y(\gamma,r,\xi) = \frac{2r}{\xi^2+1}\, ,\quad \theta(\gamma,r,\xi) = -2\, \arctan \frac 1 \xi\, .
    \end{equation}
    From \eqref{coord-cr-xytheta} it follows that the Liouville measure $dm(x,y,\theta) =y^{-2}\, dx\, dy\, d\theta$ becomes
    \[
        dm(\gamma,r,\xi) = \frac{1}{2\, r^2}\, d\gamma\, dr\, d\xi
    \]
    in the coordinates $(\gamma,r,\xi)$, and we set $d\tilde \nu = r^{-2} d\gamma dr d\xi$. Moreover by the definition of $\xi$, the horocycle flow $h_s^+$ acts in the coordinates $(\gamma,r,\xi)$ simply by translation on $\xi$, that is
    \[
        h_s^+(\gamma,r,\xi) = (\gamma,r,\xi+s)\, .
    \]
    These facts, together with $\tilde m(S\MMM) = 2\pi^2/3$, prove (ii)
    and (iii). Finally (i) follows from (ii).
\end{proof}

\noindent
Let us calculate an explicit expression for the return time
$s_h(\gamma,r,\eps)$. From Lemma \ref{lem:via-matrici} it
follows that $s_h(\gamma,r,\eps)$ is the time spent by flowing along
the horocycle until its intersection with the right representative of
$\III_0$. Now, by the identification of $S\HP$ with
$PSL(2,\R)$ (see \eqref{rep-matrici-sh}), the map $\WWW$ can be
defined also as a map from $PSL(2,\R)$ to $\WW$. In particular
\begin{equation}\label{map-sl2r-ww-inv}
(z,\zeta) = \WWW^{-1}(\gamma, r, \eps)\quad \Rightarrow\quad \Gamma_{(z,\zeta)} = \begin{pmatrix} \frac{\gamma}{\sqrt{2r}} & - \frac{r+\eps \, \sqrt{r^2-\gamma^2}}{\sqrt{2r}}\\[0.2cm] \frac{1}{\sqrt{2r}} & \frac{r-\eps \, \sqrt{r^2-\gamma^2}}{\gamma \sqrt{2r}} \end{pmatrix} \, .
\end{equation}

\begin{prop} \label{prop:return-time} For
    $(z,\zeta)= \WWW^{-1}(\gamma, r, \eps) \in \CCC$ with
    $z\in \III_0$ and $\theta(\zeta)\in (-\pi,0)$, let
    $\Gamma_{(z,\zeta)}\in PSL(2,\R)$ be as in
    \eqref{map-sl2r-ww-inv}. Assume that the first return to $C$ for
    $\tilde h_s^+$ with negative $s$ along $W^+(z,\zeta)$ occurs on
    $\Gamma^{-1}(\III_0)$ with $\Gamma\in PSL(2,\Z)$, and set
    $(\gamma',r',\eps')=\PPP_h(\gamma,r,\eps)$ and
    $(z',\zeta') = \WWW^{-1}(\gamma',r',\eps')$. Write also the corresponding
    matrix $\Gamma_{(z',\zeta')}$ in the form
    \eqref{map-sl2r-ww-inv}. Then $s_h(\gamma,r,\eps)$ satisfies
    the equation
    \[
        \begin{pmatrix} 1 & s_h(\gamma,r,\eps) \\[0.2cm] 0 & 1 \end{pmatrix} = \Gamma_{(z,\zeta)}^{-1}\, \Gamma^{-1}\, \Gamma_{(z',\zeta')}\, .
    \]
\end{prop}

\begin{proof}
The horocycle flow is represented by the right action of matrices in $PSL(2,\R)$ as in \eqref{horocycle-sl2r}. Hence
\[
\Gamma_{h_{s_h(\gamma,r,\eps)}^+(z,\zeta)} = \Gamma_{(z,\zeta)} \, \begin{pmatrix} 1 & s_h(\gamma,r,\eps) \\[0.2cm] 0 & 1 \end{pmatrix}\, .
\]
In addition, if the first return to $\CCC$ occurs on $\Gamma^{-1}(\III_0)$ at the point $\Gamma^{-1}(z',\zeta')$, we have
\[
\Gamma_{h_{s_h(\gamma,r,\eps)}^+(z,\zeta)} = \Gamma^{-1}\Gamma_{(z',\zeta')}\, .
\]
This concludes the proof.
\end{proof}

\noindent
Using the above Proposition along with Tables \ref{tab-eps+1}, \ref{tab-eps-1}, \ref{tab-eps0}, and Equation
\eqref{map-sl2r-ww-inv}, we can make the form of $s_h(\gamma,r,\eps)$ explicit for all
$(\gamma,r,\eps) \in \WW$. By means of the functions
\begin{equation}\label{funz-supporto}
u(\gamma,r,\eps) := - \eps\, \frac{\sqrt{r^2-\gamma^2}}{\gamma}\, ,\quad v(\gamma,r,\eps):= \eps\, \frac{r}{\gamma}\, , \quad \delta(\gamma,r,\eps):= \delta_0(\eps)\, ,
\end{equation}
where $\delta_0(\cdot)$ takes the value 1 at 0, and vanishes elsewhere, we get the expressions collected in the following two tables. One may readily check that $s_h(\gamma,r,\eps)<0$ for all $(\gamma,r,\eps) \in \WW$ and also $s_h \in L^1(\WW,\nu)$, so that the $(\Sigma_h, \tilde \nu)$ as defined in Proposition \ref{mis-inv-horo} is a finite measure space.

\begin{table}[h]
    \begin{tabular}{c|c}
        \textbf{Conditions on $\gamma$ and $r$} & $s_h(\gamma,r,\eps)$ \\
        \specialrule{1pt}{1pt}{1pt}
        $r<1-\gamma$ & $((u-v)\circ\PPP_h)(\gamma,r,\eps) - (u-v)(\gamma,r,\eps) + (\delta\circ\PPP_h +\delta)(\gamma,r,\eps)$\\
        \hline
        $r\ge 1-\gamma$ & $((u+v)\circ\PPP_h)(\gamma,r,\eps) - (u+v)(\gamma,r,\eps) - (\delta\circ\PPP_h +\delta)(\gamma,r,\eps)$ \\
    \end{tabular}
    \caption{The return time $s_h$ for $\eps\in \{0,+1\}$ in terms of the functions defined in \eqref{funz-supporto}.}
    \label{tab-sh-eps0+1}
\end{table}

\begin{table}[h]
    \begin{tabular}{c|c}
        \textbf{Conditions on $\gamma$ and $r$} & $s_h(\gamma,r,\eps)$ \\
        \specialrule{1pt}{1pt}{1pt}
        $r<\min\{1+\gamma, \gamma(1+\gamma)\}$ & $(u\circ \PPP_h -u)(\gamma,r,\eps)$ \\
        \hline
        $1+\gamma \le r <\gamma(1+\gamma)$ & $((u-v)\circ\PPP_h)(\gamma,r,\eps) - (u-v)(\gamma,r,\eps) + (\delta\circ\PPP_h +\delta)(\gamma,r,\eps)$ \\
        \hline
        $r\ge \gamma(1+\gamma)$ & $((u+v)\circ\PPP_h)(\gamma,r,\eps) - (u+v)(\gamma,r,\eps) - (\delta\circ\PPP_h +\delta)(\gamma,r,\eps)$ \\
    \end{tabular}
  \caption{The return time $s_h$ for $\eps=-1$ in terms of the functions defined in \eqref{funz-supporto}.}
  \label{tab-sh-eps-1}
\end{table}

\begin{rem}\label{rem-suspension}
    The expressions for the return time $s_h(\gamma,r,\eps)$ displayed in
    Table \ref{tab-sh-eps0+1} and \ref{tab-sh-eps-1} show that the
    suspension flow $\phi_s :\Sigma_h \to \Sigma_h$ over
    $(\WW,\PPP_h)$, defined in \eqref{susp-flow-horo}, has roof function
    which is not cohomologous to a constant. If it were, there should be
    a function $\Psi: \WW \to \R$ so that
\[
s_h - \int_\WW\, s_h\, d\nu = \Psi \circ \PPP_h - \Psi\, .
\]
This does not happen because the sets on which we have defined the return time are not $\PPP_h$-invariant. For example, choosing $(\gamma,r,\eps)$ with $\eps=+1$ and $r\in ((1-\gamma)(1-2\gamma), 1-\gamma)$, we see that $(\gamma',r',\eps') = \PPP_h(\gamma, r,\eps)$ satisfies $\eps'=+1$ and $r' = r/(1-\gamma)^2 >1-\gamma' = \gamma/(1-\gamma)$. On the other hand, the reader will have noticed that the function $v$ changes sign when passing from the first to the second row in Table \ref{tab-sh-eps0+1}.
\end{rem}

\section{Closed horocycles and the permuted Stern-Brocot trees} \label{sec:periodic}

We have shown in Proposition \ref{mis-inv-horo} that the suspension flow $\phi_s$ over the system $(\WW,\PPP_h)$ is isomorphic to the horocycle flow on $S\MMM$ with respective invariant measures. This implies that there is a one-to-one correspondence between the closed horocycles on $S\MMM$ which intersect $\CCC$ and the periodic points of the map $\PPP_h$. This correspondence may be stated by saying that for a point $(z,\zeta)\in \CCC$ with $(\gamma,r,\eps)=\WWW(z,\zeta)\in \WW$, there exist $\Gamma \in SL(2,\Z)$ and $s\in (-\infty,0)$ such that $(z,\zeta) = \Gamma_*(\tilde h^+_s(z,\zeta))$ if and only if there exists $n\in \N$ such that $\PPP_h^n(\gamma,r,\eps) = (\gamma,r,\eps)$.

\noindent
It is well-known (see e.g. \cite[chap. 11]{eins-book}) that the closed horocycles are generated by the points $(iy,i) \in S\MMM$. More precisely, they are the projections on $S\MMM$ of the horocycles $W^+(iy,i)$ in $S\HP$ given by the set of points $(x+iy,i)$ with $x\in \R$. Since $ h_s^+ : (x+iy,i) \to (x+sy+iy,i)$, they have length $1/y$ when projected on $S\MMM$.

\begin{prop}\label{prop:closed-w}
If $y\le \frac 12$ the projection of the horocycle $W^+(iy,i)$ on $S\MMM$ intersects $\CCC$.
\end{prop}

\begin{proof}
By applying $S$ to $W^+(iy,i)$ we see that $S(W^+(iy,i))$ is the circle tangent to $\R$ at 0 and of radius $\frac 1{2y}\ge 1$. Then $R(S(W^+(iy,i)))$, the circle tangent to $\R$ at 1 and of radius $\frac 1{2y}\ge 1$, is equal to $W^+(z,\zeta)$ with $(z,\zeta)\in \CCC$ and $z\in \III_0$. If $y<1/2$, we can choose $(z,\zeta)$ to be the determined by the intersection of $R(S(W^+(iy,i)))$ with $\III_0$ with smaller value of the variable $y$, therefore it follows that $\WWW(z,\zeta) = (1,\frac 1{2y},-1)$. If $y=1/2$, the intersection of $R(S(W^+(iy,i)))$ with $\III_0$ consists of a single point, and therefore $\WWW(z,\zeta) = (1,1,0)$.
\end{proof}

\begin{rem}\label{rem-closed-w}
In the previous Proposition, one can use the identification between $S\HP$ and $PSL(2,\R)$ recalled in Section \ref{sec:facts}, to prove that for the point $(z,\zeta)\in \CCC$ in the intersection between $R(S(W^+(iy,i)))$ and $\III_0$ we have
\[
\Gamma_{(z,\zeta)} = \begin{pmatrix}
\sqrt{y} & \frac{-1+\sqrt{1-4y^2}}{2\sqrt{y}}\\[0.2cm]
\sqrt{y} & \frac{1+\sqrt{1-4y^2}}{2\sqrt{y}}
\end{pmatrix}.
\]
Moreover $\tilde h^+_s(z,\zeta) = (z,\zeta)$ for $s=-1/y$ since
\[
SR^{-2} \begin{pmatrix}
\sqrt{y} & \frac{-1+\sqrt{1-4y^2}}{2\sqrt{y}}\\[0.2cm]
\sqrt{y} & \frac{1+\sqrt{1-4y^2}}{2\sqrt{y}}
\end{pmatrix}
\begin{pmatrix}
1 & -\frac 1y \\[0.2cm]
0 & 1
\end{pmatrix} =  \begin{pmatrix}
\sqrt{y} & \frac{-1+\sqrt{1-4y^2}}{2\sqrt{y}}\\[0.2cm]
\sqrt{y} & \frac{1+\sqrt{1-4y^2}}{2\sqrt{y}}
\end{pmatrix} .
\]
\end{rem}

\noindent
A consequence of Proposition~\ref{prop:closed-w} is that closed
horocycles of length $\ell\ge 2$ correspond to periodic points of the
map $\PPP_h$, and in particular the horocycle of length $\ell$
corresponds to the point $(1, \ell/2, \eps) \in \WW$ with $\eps=\pm 1$
if $\ell>2$, and $\eps=0$ if $\ell=2$. In the following we
characterise the periodic points of $\PPP_h$, their orbits and their
periods.

\begin{exam}\label{esempio-per}
Let us follow the orbits of two points of the form $(1,r,\eps)$ with $r\in \N$ under the action of $\PPP_h$. First, setting $r=1$, the point $(1,1,0)$ is a fixed point of $\PPP_h$ as shown in Table \ref{tab-eps0}.

\noindent
Setting instead $r=2$, we can consider the point $(1,2,-1)$. Applying $\PPP_h$ according to Table \ref{tab-eps-1} for $r=\gamma(1+\gamma)$, we find $\PPP_h(1,2,-1) = (1/2, 1/2, 0)$. Then $\PPP_h(1/2, 1/2, 0) = (2,2,0)$ and $\PPP_h(2,2,0) = (1,2,+1)$ by Table \ref{tab-eps0}. Finally we use Table \ref{tab-eps+1} to get $\PPP_h(1,2,+1) = (1,2,-1)$. Hence $(1,2,-1)$ is periodic of period 4.
\end{exam}

\noindent
First we characterise the periodic points of $\PPP_h$ without using the known results on closed horocycles.

\begin{prop}\label{prop:per-points-rat}
If a point $(\gamma,r,\eps)\in \WW$ is periodic for $\PPP_h$ then $\gamma \in \Q$.
\end{prop}

\begin{proof}
    By using the map $\WWW$ given in \eqref{eq-map-W}, one realises
    that a closed horocycle on $S\MMM$ which intersects $\CCC$ can be
    associated to a point $(\gamma,r,\eps)$ in $\WW$ with
    $\eps \in \{-1,0\}$. Indeed, if $(\gamma,r,+1) = \WWW(z,\zeta)$ is
    periodic, then also $(\gamma,r,-1)$ must be periodic, since they
    define the same closed horocycle.

\noindent
Then, looking at Table \ref{tab-eps0} it follows that if
$(\gamma,r,0)$ is periodic then either $(\gamma,r)=(1,1)$ or $(2,2)$,
and we are dealing with the fixed point or the period 4 orbit of
Example \ref{esempio-per}, or else it can be reduced to a periodic
point with $\eps=-1$.

\noindent
We are thus left to study a periodic point of the form
$(\gamma,r,-1)$. Studying its orbit and the structure of the
correspondent closed horocycle (see Figure \ref{fig-case-1}), one
readily checks that there exists $k\ge 1$ such that
$\PPP_h^k(\gamma,r,-1) = (\gamma,r,+1)$. Hence by Lemma
\ref{lem:via-matrici} there exists $\Gamma_1 \in SL(2,\Z)$ which
yields $\gamma = \Gamma_1(\gamma)$ and
$r = \DDD{\Gamma_1(\gamma)}\, r$. Moreover, since the orbit of
$(\gamma,r,-1)$ is periodic, there exists $m\ge 1$ such that
$\PPP_h^m(\gamma,r,+1) = (\gamma,r,-1)$, so that $n=m+k$ is the period
of $(\gamma,r,-1)$. Again this implies that there exists
$\Gamma_2 \in SL(2,\Z)$ which yields $\gamma = \Gamma_2(\gamma)$ and
$r = \DDD{\Gamma_2(\gamma)}\, r$.

\noindent
From Tables \ref{tab-eps+1} and \ref{tab-eps-1} it follows that $\Gamma_1$ could also be the identity, but $\Gamma_2$ cannot be such since $\Gamma_2^{-1}(\III_0)$ is the set on which one of the return along the horocycle occurs, after the return corresponding to $(\gamma,r,+1)$. This means that $\gamma$ satisfies
\[
\frac{a\gamma+b}{c\gamma+d} = \gamma\quad \text{and}\quad |c\gamma + d|=1
\]
for $a,b,c,d \in \Z$. Therefore $\gamma \in \Q$.
\end{proof}

\noindent
The following analysis of the relationship between the periodic orbits and the permuted Stern-Brocot tree $\hat{\TTT}$ provides the opposite implication, that if $\gamma \in \Q$ then $(\gamma,r,\eps)$ is periodic for all $r,\eps$.

\subsection{Periodic orbits of $\PPP_h$ and the permuted Stern-Brocot tree} \label{sec:charac-periodic}

We start by studying the orbit of the points $(1,r,-1)\in \WW$ with $r>1$. The case $r=1$ is easy since the corresponding point in $\WW$ is $(1,1,0)$ and has been already discussed in Example \ref{esempio-per}.

\noindent
From the argument in the proof of Proposition \ref{prop:per-points-rat} it follows that there exists $k\ge 1$ such that $\PPP_h^k(1,r,-1)=(1,r,+1)$.  We then apply Table \ref{tab-eps+1} to get $\PPP_h(1,r,+1)=(1,r,-1)$. Hence the points $(1,r,-1)$ are periodic for all $r>1$, and all the points in their orbits correspond to returns to $\CCC$ occurring on sets $\Gamma(\III_0)$, for $\Gamma\in SL(2,\Z)$, in $\{z\in \HP: \Real (z)\le 0\}$. In particular, let $W^+_r$ be the horocycle which is tangent to $\R$ at $1$ and has radius $r>1$. We need to determine all the sets $\Gamma^{-1}(\III_0)$ with $\Gamma\in SL(2,\Z)$ which intersect $W^+_r$ in $\{z\in \HP: \Real (z)\le 0\}$.

\noindent
The easy case is given by the lines $\III_n = R^n(\III_0)$ with $n\le 0$. It is clear that $W^+_r \cap \III_n$ is not empty if and only if $r\ge 1-n$. It turns out that it is more useful to study the intersections with the geodesics with rational end points in the intervals $[n-1,n]$ for $n\le 0$.

\noindent
For a Farey pair $p/q, p'/q' \in \Q^+$ with $p/q< p'/q'$ (see Section \ref{sec:trees}), let
\begin{equation}\label{geod-pq}
\geod\left(\frac pq,\frac{p'}{q'}\right) := \set{ x+iy \in \HP \, :\, \left( x + \frac 12 \frac pq + \frac 12 \frac{p'}{q'}\right)^2 + y^2 = \frac{1}{(2qq')^2} }
\end{equation}
be the geodesic with end points in $-p/q$ and $-p'/q'$. For example $\geod(0/1,1/1)= \JJJ_{-1}$.

\begin{lem}\label{lem:coding-geod-pq}
For a Farey pair $p/q, p'/q' \in \Q^+$ we have
\[
\geod\left(\frac pq,\frac{p'}{q'}\right) = \left(\hat{M}\left(\frac pq \oplus \frac{p'}{q'}\right)\right)^{-1}(\III_0)\, ,
\]
where $\hat{M}(a/b)\in \{L,R\}^*$ is the reverse word associated to the rational number $a/b$ as in Section \ref{sec:sb-perm-tree}, and every symbol in $\hat{M}(a/b)$ is interpreted as a matrix of $SL(2,\Z)$ acting on $\HP$.
\end{lem}

\begin{proof}
By definition of the Stern-Brocot sets, the half-circles $\geod\left(p/q,p'/q'\right)$ are contained between two lines $\III_{n-1}$ and $\III_n$ for an integer $n\le 0$.

\noindent
Let us first consider the case $n=0$, that is $p/q, p'/q' \in [0,1]$. In particular the Farey sum $p/q \oplus p'/q'$ is in $[0,1]$, and for each $m/s \in [0,1]$ with $m,s$ coprime, there is exactly one couple $p/q, p'/q'$ such that $m/s = p/q \oplus p'/q'$. We argue by induction on the level of $m/s$ in the Stern-Brocot tree $\TTT$. The minimum level for $m/s$ is one, and this is the level only of the fraction $1/2= 0/1\oplus 1/1$. In this case the result holds, since
\[
\geod\left(\frac 01,\frac{1}{1}\right) = \JJJ_{-1} = L^{-1}(\III_0)
\]
and $M(1/2) = \hat{M}(1/2) = L$. Note that the geodesic $\geod(0/1,1/1)$ has ends in common with the geodesics $\geod(0/1,1/2)$ and $\geod(1/2,1/1)$, and the last are the only geodesics of the family of geodesics $\geod\left(p/q,p'/q'\right)$ with $p'q-q'p=1$. The two geodesics satisfy
\[
\hat{M}\left(\frac 12 \right) \geod\left(\frac 01,\frac{1}{2}\right) = \geod\left(\frac 01,\frac{1}{1}\right)\quad \text{and}\quad \hat{M}\left(\frac 12 \right) \geod\left(\frac 12,\frac{1}{1}\right) = \III_{-1}\, .
\]
Let us assume now that $m/s$ is in the $(k+1)$-th level of $\TTT$ and the result holds for all fractions on all the previous levels. Then there exist $p/q, p'/q'$ on two different previous levels in $\TTT$ and satisfy $m/s = p/q \oplus p'/q'$. Without loss of generality we assume that $p/q\in \TTT_{\ell}$ and $p'/q'\in \TTT_{\ell'}$ with $\ell > \ell'$, then by the construction of $\TTT$ we have $\ell=k$. It follows that $M(m/s) = M(p/q)R$ since $p/q<m/s<p'<q'$. In addition there exists $p''/q'' \in \TTT$ such that $p/q=p''/q'' \oplus p'/q'$, hence the geodesics $\geod(p/q, p'/q')$ and $\geod(p''/q'', p'/q')$ have one common end point, and the first one lies below the second one in $\HP$. By the inductive assumption we have
\[
\III_0 = \hat{M} \left(\frac{p''}{q''}\oplus \frac{p'}{q'}\right) \geod\left(\frac{p''}{q''},\frac{p'}{q'}\right)= \hat{M} \left(\frac pq \right) \geod\left(\frac{p''}{q''},\frac{p'}{q'}\right)
\]
and
\[
\hat{M} \left(\frac pq \right) \geod\left(\frac{p}{q},\frac{p'}{q'}\right) = \III_{-1}
\]
because $\geod(p/q, p'/q')$ and $\geod(p''/q'', p'/q')$ have their leftmost end point $p'/q'$ in common. This is the analogue of what has been shown above for $\geod(0/1,1/1)$ and $\geod(1/2,1/1)$. Then
\[
\III_0 = R\, \hat{M} \left(\frac pq \right) \geod\left(\frac{p}{q},\frac{p'}{q'}\right) = \hat{M} \left(\frac rs \right) \geod\left(\frac{p}{q},\frac{p'}{q'}\right) = \hat{M}\left(\frac pq \oplus \frac{p'}{q'}\right)  \geod\left(\frac{p}{q},\frac{p'}{q'}\right)
\]
and the statement is proved for the fractions on the $(k+1)$-th level of $\TTT$ which are in $[0,1]$.

\noindent
The case $n\le -1$ follows by noting that if $p/q, p'/q' \in [|n|,|n|+1]$, then
\[
R^{|n|} \geod\left(\frac{p}{q},\frac{p'}{q'}\right) =  \geod\left(\frac{p}{q}-|n|,\frac{p'}{q'}-|n| \right)
\]
with $p/q-|n|,p'/q'-|n| \in [0,1]$. By the previous case we obtain
\[
\hat{M}\left(\frac{p-q|n|}{q} \oplus \frac{p'-q'|n|}{q'}\right) \geod\left(\frac{p}{q}-|n|,\frac{p'}{q'}-|n| \right) = \III_0
\]
hence
\[
\hat{M}\left(\frac{p-q|n|}{q} \oplus \frac{p'-q'|n|}{q'}\right) R^{|n|} \geod\left(\frac{p}{q},\frac{p'}{q'}\right) = \III_0\, .
\]
Since from the structure of $\TTT$ we have $M(m/s) = R^k M(m/s-k)$ for all $m/s \in [k,k+1)$, we have
\[
\hat{M}\left(\frac{p-q|n|}{q} \oplus \frac{p'-q'|n|}{q'}\right) R^{|n|} = \hat{M}\left(\left(\frac pq \oplus \frac{p'}{q'}\right)-|n| \right) R^{|n|} = \hat{M}\left(\frac pq \oplus \frac{p'}{q'}\right)
\]
and the statement is proved.
\end{proof}

\noindent
From the previous lemma it follows that the geodesics $\geod(p/q,p'/q')$ may be arranged in a tree which corresponds to the Stern-Brocot tree $\TTT$ by the bijection
\[
\frac ms = \frac pq \oplus \frac{p'}{q'} \quad \leftrightarrow \quad \geod\left(\frac{p}{q},\frac{p'}{q'}\right)
\]
with the convention that $1/1$ corresponds to $\III_0 = \geod(0/1,1/0)$. If $m/s$ is a daughter of $m'/s'$ in $\TTT$, then the two corresponding geodesics have one end point in common, given by the common parent of $m/s$ and $m'/s'$, and the geodesic corresponding to $m/s$ lies in $\HP$ below that of $m'/s'$.

\noindent
At the same time, the tree of geodesics is in bijection with the permuted Stern-Brocot tree $\hat{\TTT}$, since for $\geod(p/q,p'/q')$ the word $\hat{M}(p/q \oplus p'/q')$ gives the position of $m/s = p/q \oplus p'/q'$ in $\hat{\TTT}$. This relation plays a fundamental role when studying the action of the map $\PPP_h$ thanks to Lemma \ref{lem:via-matrici}.

We now give an elementary result on the intersection of $\geod(p/q,p'/q')$ with the horocycle $W^+_r$ tangent to $\R$ at 1 and of radius $r>1$.

\begin{lem}\label{lem:inters-geod-horo}
For a Farey pair $p/q, p'/q' \in \Q^+$, the horocycle tangent to $\R$ at $\rho>0$ and of radius $r\ge \rho$ intersects $\geod\left(p/q,p'/q'\right)$ if and only if $r \ge (q\rho+p)(q'\rho+p')$.
\end{lem}

\begin{proof}
    We look for an equivalent condition to the existence of solutions
    to the system given by the polynomial equation in \eqref{geod-pq}
    describing $\geod(p/q,p'/q')$ and the polynomial equation
    $(x-\rho)^2+y^2-2ry=0$ describing the horocycle tangent to $\R$ at
    $\rho>0$ and of radius $r\ge \rho$. The two circles have a
    tangency point if and only if
\[
\left( r+\frac{1}{2qq'}\right)^2 = r^2 + \left( \rho + \frac 12 \left(\frac pq + \frac{p'}{q'}\right)\right)^2
\]
which, using that $p'q-pq'=1$ and elementary computations, is
equivalent to $r=(q\rho+p)(q'\rho+p')$. Then the two circles intersect
if and only if the radius $r$ is greater than or equal to the values for
which they are tangent.
\end{proof}

\noindent
We are now ready to state the main result of this section, which provides a precise characterisation of the orbits under $\PPP_h$ of the periodic points $(1,r,-1)$ for all $r\ge 1$ and of all the periodic points of $\PPP_h$.

\begin{defn}\label{def:energy-pq}
Given $p/q\in \Q^+$ with $p,q$ coprime, we call \emph{energy} of $p/q$ the number $\xi(p/q):= pq$.
\end{defn}

\begin{thm}\label{thm-periodic}
For all $r\ge 1$, the point $(1,r,-1)\in \WW$ is periodic for $\PPP_h$ and has period $\per (r)$ given by
\[
    \per (r) = 2\cdot \# \left\{ a,b\in \N\, :\, (a,b)=1, ab< r\right\} +
    \# \left\{ a,b\in \N\, :\, (a,b)=1,\ ab= r\right\}.
\]
In addition, a point $(\gamma, r, \eps)\in \WW$ is periodic if and only if $\gamma\in \Q$, and in this case if $\gamma = a/b$ with $a,b$ coprime, it is in the orbit of the point $(1, rb^2, -1)$.
\end{thm}

\begin{proof}
Let us consider a point $(1,r,-1)\in \WW$ with $r\ge 1$. Note that the case $r=1$ has been already discussed in Example \ref{esempio-per}, from which $\per (1)=1$.
Then let $r>1$. We have seen that $(1,r,-1)$ is periodic and that $(1,r,\pm 1)$ are in the same orbit. Moreover, all the other points in the orbit of $(1,r,-1)$ are given by the intersections of $W^+_r$, the horocycle tangent to $\R$ at 1 and of radius $r$, with the sets $\Gamma^{-1}(\III_0)$ for some $\Gamma \in SL(2,\Z)$.
In particular, by Lemma \ref{lem:via-matrici}, if $W^+_r$ intersects $\Gamma^{-1}(\III_0)$ then the point $(\gamma',r',\eps')$, with $\gamma' = \Gamma(1)$ and $r'= \DDD{\Gamma(1)}r$, is in the orbit of $(1,r,-1)$ for $\PPP_h$. From this it follows first of all that a first class of points in the orbit of $(1,r,-1)$ are given by the intersections of $W^+_r$ with $\III_n=R^n(\III_0)$ for $1-r\le n< 0$. If $n=1-r$ there is only one intersection, which gives $\eps'=0$ for the corresponding point, and if $n>1-r$ there are two intersections giving $\eps'=\pm 1$ for the point in the orbit. By Lemma \ref{lem:via-matrici} the sets
\begin{equation} \label{transl-orbit-per}
\left\{ \left(1+|n|,r,\pm 1\right) \, :\, n\in \Z^-\, ,\, 1\le |n|< r-1\right\} \cup \left\{ \left(1+|n|,r,0\right) \, :\, n\in \Z^-\, ,\, |n|= r-1\right\}
\end{equation}
are made of points in the orbit of $(1,r,-1)$. In the formula for the period of $(1,r,-1)$, together with $(1,r,\pm 1)$ they correspond to all the points with $p\in \N$ and $q=1$.

\noindent
The other class of points in the orbit of $(1,r,-1)$ is given by the intersections of $W^+_r$ with the geodesics $\geod(p/q,p'/q')$ defined in \eqref{geod-pq}. It is an elementary observation that, together with $\III_n$, these geodesics cover all the sets $\Gamma^{-1}(\III_0)$ for some $\Gamma \in SL(2,\Z)$ in $\HP \cap \{\Real (z)<0\}$. By Lemma \ref{lem:inters-geod-horo} the horocycle $W^+_r$ intersects $\geod(p/q,p'/q')$ if and only if $r\ge (p+q)(p'+q')$, and by Lemma \ref{lem:coding-geod-pq} the intersection corresponds to the point $(\gamma',r',\eps)$ with
\[
\gamma'= \hat{M}\left(\frac pq\oplus \frac{p'}{q'}\right) (1), \quad r' = \DDD{\hat{M}\left(\frac pq\oplus \frac{p'}{q'}\right)(1)} r,
\]
with $\eps'=0$ if the two circles are tangent, $\eps'=\pm 1$ if there are two intersections. Let $m/s = p/q \oplus p'/q'$. Then $\gamma' = \hat{M}(m/s)(1)$ is the fraction that in $\hat{\TTT}$ occupies the place of $m/s$ in $\TTT$.

\noindent
Using the notations introduced in Section \ref{sec:sb-tree} and Lemma \ref{lem:matr-perm} we have
\[
M\left(\frac pq\oplus \frac{p'}{q'}\right) = \begin{pmatrix} p' & p \\ q' & q
\end{pmatrix} \quad \text{and} \quad \hat{M}\left(\frac pq\oplus \frac{p'}{q'}\right) = \begin{pmatrix} q & p \\ q' & p'
\end{pmatrix}.
\]
In particular, setting $\gamma'=m'/s'$, it follows that
\[
\gamma' = \frac{m'}{s'} = \frac{p+q}{p'+q'}.
\]
Moreover, using Definition \ref{def:energy-pq} of the energy of a positive rational number, we obtain
\begin{equation} \label{energy-orbit-per}
\xi\left(\gamma' \right) = (p+q)(p'+q')\quad \text{and} \quad \DDD{\hat{M}\left(\frac pq\oplus \frac{p'}{q'}\right)(1)} = (p'+q')^{-2} = (s')^{-2}\, .
\end{equation}
We can sum up the consequences of \eqref{transl-orbit-per} and \eqref{energy-orbit-per} by stating that a rational number $a/b$, with $(a,b)=1$, is the first component of a point in the orbit of $(1,r,-1)$ under $\PPP_h$ if and only if $r$ is greater or equal than the energy $\xi(a/b)$, and more precisely the orbit of $(1,r,-1)$ is given by
\[
\set{ \left( \frac ab, \frac{r}{b^2}, \eps\right)\, :\, (a,b)=1\, ,\, ab < r\, ,\, \eps\in \{-1,+1\}} \cup \set{ \left( \frac ab, \frac{r}{b^2}, 0\right)\, :\, (a,b)=1\, ,\, ab = r}.
\]
The first part of the statement of the theorem is proved. Proposition \ref{prop:per-points-rat} and the characterisation of the orbit of $(1,r,-1)$ readily imply the second part of the statement.
\end{proof}

\noindent
The result of Theorem \ref{thm-periodic} may be refined by studying the dynamical ordering of the points in the orbit of $(1,r,-1)$. By looking at the intersections of a horocycle $W^+_r$ with the geodesics $\geod(p/q,p'/q')$ defined in \eqref{geod-pq}, and at the ordering of these intersections described by the relative geometrical positions of the geodesics in $\HP$, we can produce an algorithm to determine the points of the periodic orbit of $(1,r,-1)$ in their dynamical order up to $(1,r,+1)$ by using the permuted Stern-Brocot tree $\hat{\TTT}$. A description of this algorithm is given in Appendix \ref{app:algorithm}.

\begin{rem}
    Adopting an alternative point of view, we can look at the period
    $\per(r)$ of the $\PPP_h$-periodic orbit containing
    $(1,r,-1)$ as the number of intersections of the horocycle
    $W^+(iy,i)=\{ x+iy\, :\, x \in \R \}$, with $y=1/(2r)$, with
    the (geodesic) sides of the triangles forming the Farey
    tassellation of the domain
    $\set{x+iy \in \HP \,:\, 0\leq x \leq 1}$ (see
    Figure~\ref{fig-fareytass} below).

    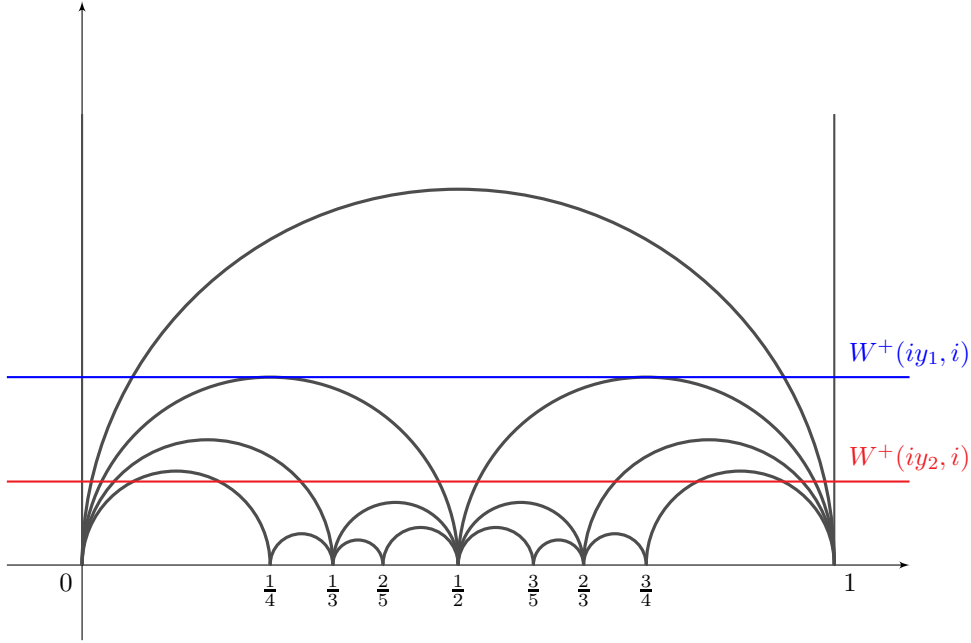
\begin{figure}[h!]
        \pgfmathsetmacro{\myxlow}{0}
\pgfmathsetmacro{\myxhigh}{1}

\begin{tikzpicture}[scale=10]
    \draw[-latex',thin](\myxlow-0.1,0) -- (\myxhigh+0.1,0);
    \draw[-latex',thin](0,-0.1) -- (0,0.75);

    \draw (0,0)   -- (0  ,-0.00) node[below left]{$0$};
    \draw (1/4,0) -- (1/4,-0.00) node[below]{$\frac 14$};
    \draw (1/3,0) -- (1/3,-0.00) node[below]{$\frac 13$};
    \draw (2/5,0) -- (2/5,-0.00) node[below]{$\frac 25$};
    \draw (1/2,0) -- (1/2,-0.00) node[below]{$\frac 12$};
    \draw (3/5,0) -- (3/5,-0.00) node[below]{$\frac 35$};
    \draw (2/3,0) -- (2/3,-0.00) node[below]{$\frac 23$};
    \draw (3/4,0) -- (3/4,-0.00) node[below]{$\frac 34$};
    \draw (1,0)   -- (1  ,-0.00) node[below right]{$1$};

    \draw[thick,farey] (0,0.6) -- (0,0);
    \draw[thick,farey] (1,0.6) -- (1,0);

    \draw[very thick,farey] (1,0) arc(0:180:1/2);

    \draw[very thick,farey] (1/2,0) arc(0:180:1/4);
    \draw[very thick,farey] (1/1,0) arc(0:180:1/4);

    \draw[very thick,farey] (1/3,0) arc(0:180:1/6);
    \draw[very thick,farey] (1/2,0) arc(0:180:1/12);
    \draw[very thick,farey] (2/3,0) arc(0:180:1/12);
    \draw[very thick,farey] (1/1,0) arc(0:180:1/6) ;

    \draw[very thick,farey] (1/4,0) arc(0:180:1/8) ;
    \draw[very thick,farey] (1/3,0) arc(0:180:1/24);
    \draw[very thick,farey] (2/5,0) arc(0:180:1/30);
    \draw[very thick,farey] (1/2,0) arc(0:180:1/20);
    \draw[very thick,farey] (3/5,0) arc(0:180:1/20);
    \draw[very thick,farey] (2/3,0) arc(0:180:1/30);
    \draw[very thick,farey] (3/4,0) arc(0:180:1/24);
    \draw[very thick,farey] (1/1,0) arc(0:180:1/8) ;

    \draw[thick, blue] (-0.1,1/4) -- (1.1,1/4) node[above]{$W^+(iy_1,i)$};
    \draw[thick, Red] (-0.1,1/9) -- (1.1,1/9) node[above]{$W^+(iy_2,i)$};

\end{tikzpicture}
        \caption{Farey tassellation of the domain
          $\set{x+iy \in \HP \,:\, 0\leq x \leq 1}$ up to level $4$ as
          defined in Section~\ref{sec:trees}. We also show two
          horizontal horocycles $W^+(iy_k,i)$ with $y=1/(2r_k)$, $k=1,2$: the blue horocycle has $r_1=2$ and is
          tangent to some of the geodesic sides of the triangles of
          the Farey tassellation; the red horocycle has
          $4<r_2<5$.}\label{fig-fareytass}
    \end{figure}

\noindent
In particular, the horocycle $W^+(iy,i)$ intersects the geodesic joining a Farey pair $p/q, p'/q' \in [0,1]$, with $p/q < p'/q'$, if and only if $y\leq 1/(2qq')$. Note moreover that given a Farey pair $p/q,\, p'/q' \in [0,1]$, the symmetric pair $(q-p)/q,\, (q'-p')/q'$ is in $[0,1]$ and yields a geodesic side of equal ``height" ($=1/(2qq')$), and there are no other pairs with this property. By the way, this symmetry accounts for the corresponding symmetry of the walks associated to the dynamics of the periodic points of $\PPP_h$ on the permuted Stern-Brocot tree (see Appendix \ref{app:algorithm}).
The formula for the period can thus be rewritten as
\[
    \per(r) = 2\cdot \# \left\{ \frac pq,\, \frac{p'}{q'} \, \text{Farey pairs in} \; [0,1]\, :\, qq'< r\right\} + \# \left\{ \frac pq,\, \frac{p'}{q'} \, \text{Farey pairs in} \; [0,1]\, :\, qq'= r\right\}.
\]
To further clarify the connection with the construction made in the proof of Theorem \ref{thm-periodic}, we observe that the projection on $S\MMM$ of the geodesic joining a Farey pair $p/q, p'/q'$, via RS, is the geodesic joining $P/Q, P'/Q'$ with $P=q-p$ and $Q=p$ (and similar primed relations), so that $(P+Q)(P'+Q')=qq'$.

\noindent
Finally, this picture would also provide a sort of duality between periodic horocycles and scattering geodesics (i.e. vertical geodesics landing at rational points),  in the sense that, in order to know the corresponding paths on the tree, in the first case we have to look at the intersections (of the horocycle) with Farey triangles (arcs of geodesics), while in the second case we have to look at the intersections (of the scattering geodesics) with Ford circles (horocycles).
\end{rem}

\subsection{Distribution of periodic orbits} \label{sec:equid}
According to Theorem \ref{thm-periodic}, the periodic orbits of
$\PPP_h$ are nothing more than the orbits of the one-parameter family
of points $\{(1,R,-1)\}_{R\in (1,\infty)}$ in $\WW$. Denoting by
$\OOO(1,R,-1)$ the orbit of $(1,R,-1)$, which consists of $\per(R)$
points, we first prove the following proposition.

\begin{prop}\label{prop:asymp-per} The period $\per(R)$ of the $\PPP_h$-periodic orbits $\OOO(1,R,-1)$  satisfies
\[
\per(R) = \frac{12}{\pi^2}\, R\, \log\, R + O(R) , \quad R\to \infty .
\]
\end{prop}

\begin{proof}
By the formula for the period $\per(R)$ of the point $(1,R,-1)$ given in Theorem \ref{thm-periodic}, we have to estimate the asymptotic behaviour of the cardinality of the set $\Pi(R):= \set{a,b \in \N\, :\, (a,b)=1,\, ab\le R}$,
that is
\[
\# \Pi(R) =  \sum_{n=1}^R\, \sum_{(a,b)=1, \, ab=n}\, 1
\]
Now, the inner sum counts the number of ways in which one can split the set of prime numbers dividing $n$ into two disjoint sets. This in turn is equal to the number of square-free divisors of $n$. It is a standard result in number theory (see e.g. \cite{apostol}) that  is obtained by summing over the divisors on $n$, denoted $d|n$, the square of the M\"obius function $\mu(\cdot)$, that is $\mu^2(d)=1$ if $d|n$ and it is square-free, and $\mu^2(d)=0$ if $d|n$ but is not square-free. Therefore, using standard techniques in number theory, we can write
\[
\begin{aligned}
\# \Pi(R) = & \sum_{n=1}^R\,  \sum_{d|n}\, \mu^2(d)  = \sum_{n=1}^R\, \mu^2(n)\, \left\lfloor \frac Rn \right\rfloor = \sum_{n=1}^R\, \mu^2(n)\, \frac Rn + O\left( \sum_{n=1}^R\, \mu^2(n)\right) = \\[0.3cm]
= &  \sum_{n=1}^R\, \mu^2(n)\, \frac Rn + O(R) =  \frac{6}{\pi^2}\, R\, \log R + O(R).
\end{aligned}
\]
Since $\per(R) = 2\, \# \Pi(R) + O(R)$, the proof is finished.
\end{proof}

\noindent
Furthermore, a more refined number-theoretic argument allows us to conclude that the periodic orbits are equidistributed with respect to the $\PPP_h$-invariant measure $\nu$. Although this is not surprising, after the results in \cite{sarnak,stromberg,ath-che}, we believe it is interesting as a further case of equidistribution of the periodic orbits for a dynamical system with infinite invariant measure (see \cite{heersink}). In particular, in the infinite measure settings, one expects that the equidistribution holds for a family of measures not uniformly supported on the periodic orbits. We state our result as in the classical Hopf Ratio Ergodic Theorem, that is by looking at the ratio of the distribution into two different subsets, but the proof contains also the correct speed of convergence (see \eqref{res}), which is slower than the cardinality of the periodic orbits.

\begin{thm}\label{equid-perorb}
Let $\nu$ be the $\PPP_h$-invariant measure on $\WW$ with density $k(\gamma,r,\eps)=r^{-2}$. Then for all $A,B \subseteq \WW$ measurable sets with $0<\nu(A),\nu(B)<\infty$, we have
\[
\lim_{R\to \infty}\, \frac{\sum_{(\gamma,r,\eps)\in \OOO(1,R,-1)} \, \chi_{_A}(\gamma,r,\eps)}{\sum_{(\gamma,r,\eps)\in \OOO(1,R,-1)} \, \chi_{_B}(\gamma,r,\eps)} = \frac{\nu(A)}{\nu(B)}\, .
\]
\end{thm}

\begin{proof}
We know from Theorem \ref{thm-periodic} that
\[
\OOO(1,R,-1) = \set{\left(\frac ab, \frac{R}{b^2}, \pm 1\right) \, :\, (a,b)=1\, ,\, ab<R} \cup \set{\left(\frac ab, \frac{R}{b^2}, 0\right) \, :\, (a,b)=1\, ,\, ab=R}\, .
\]
We study the distribution of the points of the orbits $\OOO(1,R,-1)$ as $R\to \infty$ for a set of the form
\begin{equation}\label{insieme-A}
A = \set{ \left(\gamma, r, \eps\right) \in \WW\, :\, s<\gamma <t\, ,\, u<r<v\, ,\, \eps\not= -1}
\end{equation}
with $s,t,u,v\in \R^+$. The choice $\eps\not=-1$ in $A$ does not reduce the generality of the result as it is evident from the structure of points in $\OOO(1,R,-1)$ and the fact that the set $\WW \cap \set{\eps=0}$ has vanishing $\nu$-measure. Moreover if the result holds for all open rectangles in $(\gamma,r)$ it holds for all measurable sets. Finally, it is enough to prove the result for $0<s<t< 1$, because if a point with $\gamma=a/b$ is in $\OOO(1,R,-1)$ then also a point with $\gamma=b/a$ is in $\OOO(1,R,-1)$, and in the $r$-components of the two points the roles of $a$ and $b$ are exchanged. Hence any set in $\WW$ can be split in two subsets, one with $\gamma$-component in $(0,1)$ and the other with $\gamma$-component in $(1,\infty)$, and it is enough to prove the result for one of these.

\noindent
 Let then $s,t,u,v\in \R^+$ with $0<s<t\le 1$ and $A$ as in \eqref{insieme-A}. The set
\[
\Pi_A(R) := \set{\frac ab\, :\, (a,b)=1\, ,\, ab\le R\, ,\, s<\frac ab <t\, ,\, u<\frac R{b^2} < v}
\]
characterizes the $\gamma$-component of the points in $A\cap \OOO(1,R,-1)$, so that
\[
\sum_{(\gamma,r,\eps)\in \OOO(1,R,-1)} \, \chi_{_A}(\gamma,r,\eps) = \#\, \Pi_A(R).
\]
First of all, $\Pi_A(R)$ can be written as
\[
\Pi_A(R) = \set{\frac ab\, :\, (a,b)=1\, ,\, ab\le R\, ,\, sb<a <bt\, ,\, \sqrt{\frac Rv}<b < \sqrt{\frac Ru}}
\]
so that the order of $s,t,u,v$ in $\R^+$ changes the description of the set. One realizes that it is enough to consider the cases
\[
s< t = u < v \quad \text{and} \quad s \le u < v = t \, .
\]
as the sets satisfying any of the other possibilities can be reduced to the union of sets falling in one of these two cases. Anyway the computations are similar and the result follows by applying the same ideas.

\noindent
We are now ready to study the asymptotic behaviour of $\# (A\cap \OOO(1,R,-1))$ as $R\to \infty$ for $A$ defined as in \eqref{insieme-A} with $0<s< t = u < v$ and $t<1$. Note that $A$ is a rectangle in $\WW$ as the condition $r\ge \gamma$ is satisfied for all points in $A$. Hence
\begin{equation}\label{nu(A)}
\nu(A) = \int_s^t \left( \int_u^v\, r^{-2}\, dr \right) \, d\gamma = (t-s)\left(\frac 1u-\frac 1v\right)\, .
\end{equation}
We have
\[
\#\, \Pi_A(R) = \sum_{\sqrt{R/v}\leq b\leq \sqrt{R/u}}\, \sum_{ \substack{(a,b)=1\\ bs\leq  a \leq bt}}\, 1 \, .
\]
The inner sum can be expressed in terms  of the \emph{Legendre totient function}
$\phi(x,n)$, a generalisation of the Euler totient function which counts the number of positive integers less or equal than $x$ which are prime to $n$. We get
\begin{equation}\label{Pi_A(R)}
\#\, \Pi_A(R) = \sum_{\sqrt{R/v}\leq b\leq \sqrt{R/u}}\, \left( \phi(bt,b) - \phi(bs,b) \right)\, .
\end{equation}
The Legendre totient function has been studied less extensively than the Euler function. We refer the reader to \cite{surya}  where estimates of $\phi(x,n)$ are given in terms of $\phi(n)$, such as
\begin{equation}\label{surya-eq}
\phi(x,n) = \frac xn\, \varphi(n) + O\left(\sum_{d|n}\, \mu^2(d) \right)\, .
\end{equation}
Moreover, as shown in the proof of Proposition \ref{prop:asymp-per}, we have
\[
\sum_{n=1}^N\, \left(\sum_{d|n}\, \mu^2(d) \right) = \frac{6}{\pi^2}\, N\, \log N + O(N) , \quad N\to \infty
\]
Using this estimate in \eqref{surya-eq} and then inserting the result in \eqref{Pi_A(R)} we get
\[
\#\, \Pi_A(R) = (t-s)\, \sum_{\sqrt{R/v}\leq b\leq \sqrt{R/u}}\,  \varphi(b)  + O\left(\sqrt{R}\, \log R\right)\, .
\]
On the other hand, a classical result in number theory (see \cite{apostol}) states that:
\begin{equation}\label{euler-1}
\sum_{n=1}^N\, \varphi(n) = \frac{3}{\pi^2}\, N^2 + O(N\, \log N)
\end{equation}
and therefore
\[
\#\, \Pi_A(R) = \frac{3}{\pi^2}\, R\, (t-s) \left(\frac 1u-\frac 1v\right) + O\left(\sqrt{R}\, \log R\right)\, .
\]
Comparing with \eqref{nu(A)} have thus proved that in this first case
\begin{equation} \label{res}
\sum_{(\gamma,r,\eps)\in \OOO(1,R,-1)} \, \chi_{_A}(\gamma,r,\eps) = \frac{3}{\pi^2}\, R\, \nu(A) + O\left(\sqrt{R}\, \log R\right)\, .
\end{equation}

\noindent
Let us now study the asymptotic behaviour of $\# (A\cap \OOO(1,R,-1))$ as $R\to \infty$ for the second case, namely for a set $A$ defined as in \eqref{insieme-A} with $0<s\le u <v = t$ and $t<1$, which is a trapezoid in $\WW$, and
\begin{equation}\label{nu(A)-2}
\nu(A) = \int_u^v \left( \int_s^r\, r^{-2}\, d\gamma \right) \, dr = \log \frac vu -s \left(\frac 1u-\frac 1v\right)\, .
\end{equation}
In this case we find
\[
\#\, \Pi_A(R) =   \sum_{ \sqrt{R/v}\leq b\leq \sqrt{R/u} }\, \sum_{ \substack{ (a,b)=1 \\ bs \leq  a \leq R/b }}\, 1
\]
so that using again the Legendre totien function and \eqref{surya-eq} we get
\[
\#\, \Pi_A(R) = \sum_{ \sqrt{R/v}\leq b\leq \sqrt{R/u} }\,\left(\phi\left(\frac Rb, b\right) - \phi(bs,b) \right) = \sum_{ \sqrt{R/v}\leq b\leq \sqrt{R/u} }\,\left(\frac R{b^2}\, \varphi(b) - s\, \varphi(b) \right) + O\left(\sqrt{R}\, \log R\right)\, .
\]
Using \eqref{euler-1} along with \cite{apostol},
\[
\sum_{n=1}^N\, \frac{\varphi(n)}{n^2} = \frac{6}{\pi^2}\, \log N + O(1)
\]
we obtain
\[
\#\, \Pi_A(R) = \frac{3}{\pi^2}\, R\, \left( \log \frac Ru - \log \frac Rv\right) - \frac{3}{\pi^2}\, R\, s\,  \left(\frac 1u-\frac 1v\right) + O\left(\sqrt{R}\, \log R\right)\, .
\]
Comparing with \eqref{nu(A)-2} we see that  \eqref{res} holds also in this case.

\noindent
In conclusion, we have proved that all the sets $A$ as in \eqref{insieme-A} satisfy \eqref{res}. This immediately implies the statement of the theorem for all couples of sets $A,B$ of this form. The extension to all measurable sets with finite $\nu$-measure is straighforward.
\end{proof}

\section{Non-periodic orbits} \label{sec:non-periodic}

In this section we study the map $\PPP_h$ on non-periodic orbits. In particular we introduce a ``time leap'' for $\PPP_h$ to rewrite it as a map from $X:=([1,\infty) \times [1,\infty) \times \{0,+1\}) \cap \WW$ to itself. We need the following result.

\begin{prop} \label{prop-repar} For the map $\PPP_h:\WW \to \WW$ we
    have:
    \begin{enumerate}
      \item for each $(\gamma,r,\eps)\in \WW$ with $\gamma \in (0,1)$ and $\eps\not= -1$, there exists $n_1(\gamma,r)$ such that $(\gamma', r', \eps') = \PPP_h^{n_1(\gamma,r)}(\gamma,r,\eps)$ satisfies $\gamma' \in [1,\infty)$;

      \item for each $(\gamma,r,\eps)\in \WW$ with $\eps= -1$, there exists $n_2(\gamma,r)$ such that $\PPP_h^{n_2(\gamma,r)}(\gamma,r,-1) = (\gamma, r, +1)$, and
        \[
            \begin{aligned}
                n_2(\gamma,r) &=  \, 2\cdot \# \left\{ a,b \in \N\, :\, (a,b)=1, (q\gamma +p)(q'\gamma +p')< r\right\} + \\ &~\qquad\qquad + \# \left\{ a,b \in \N\, :\, (a,b)=1, (q\gamma +p)(q'\gamma +p')= r\right\} -1
            \end{aligned}
        \]
        where
        \[
            M\left( \frac ab \right)= \begin{pmatrix} q & p \\ q' & p' \end{pmatrix};
        \]
      \item for each $(\gamma,r,\eps)\in \WW$ with $\gamma\not\in \N$ and $\eps\not= -1$, we have $\PPP_h^{\lfloor \gamma \rfloor}(\gamma,r,\eps)= (\gamma - \lfloor \gamma \rfloor, r, \eps)$;

      \item for each $(\gamma,r,\eps)\in \WW$ with $\gamma\in \N$, $\gamma\ge 2$ and $\eps\not= -1$, we have $\PPP_h^{(\gamma-1)}(\gamma,r,\eps) = (1,r,+1)$.
    \end{enumerate}
\end{prop}

\begin{proof}
    (i) Let $(\gamma_0,r_0,\eps_0) := (\gamma,r,\eps) \in \WW$ with
    $\gamma_0\in (0,1)$ and $\eps\not= 1$ and use the notation
    $(\gamma_n, r_n, \eps_n) = \PPP_h^n(\gamma,r,\eps)$ for all
    $n\ge 1$. To obtain $(\gamma_1,r_1,\eps_1)$ we use Table
    \ref{tab-eps+1} or Table \ref{tab-eps0}. Let $\eps_0=+1$. If
    $r_0\ge 1-\gamma_0$ we obtain $\gamma_1 = 1/(1-\gamma_0) >1$ and
    we are done with $n_1(\gamma,r)=1$. If $r_0<1-\gamma_0$, we obtain
    $(\gamma_1,r_1,\eps_1) =
    (\gamma_0/(1-\gamma_0),r_0/(1-\gamma_0)^2,+1)$, hence either
    $\gamma_1 \ge 1$ and we are done with $n_1(\gamma, r)=1$, or
    $\gamma_1\in (\gamma_0, 1)$ and we consider
    $(\gamma_2,r_2,\eps_2)$. In the case $\eps_0=0$, when
    $\gamma_0=r_0\ge 1/2$ then as before
    $\gamma_1 = 1/(1-\gamma_0) >1$ and we are done with
    $n_1(\gamma,r)=1$. If $\gamma_0=r_0< 1/2$ we obtain again
    $(\gamma_1,r_1,\eps_1) =
    (\gamma_0/(1-\gamma_0),r_0/(1-\gamma_0)^2,+1)$, now
    $\gamma_1\in (\gamma_0,1)$ and we consider
    $(\gamma_2,r_2,\eps_2)$. By repeating the argument, if we are not
    done at the $n$-th step, we have
    $\gamma_n = \gamma_0/(1-n \gamma_0) \in (\gamma_{n-1},1)$,
    $r_n = r_0/(1-n \gamma_0)^2$ and $\eps_n = +1$. At this point, if
    $r_n\ge 1-\gamma_n$, which is equivalent to
    $r_0 \ge (1-(n+1)\gamma_0)(1-n\gamma_0)$, then
    $\gamma_{n+1} = 1/(1-\gamma_n) >1$. Instead, if $r_n<1-\gamma_n$,
    we have
    $\gamma_{n+1} = \gamma_n/(1-\gamma_n) =
    \gamma_0/(1-(n+1)\gamma_0)$ and again there are two cases: either
    $\gamma_{n+1} \ge 1$ and we are done, or $\gamma_{n+1}< 1$ and we
    proceed with the $(n+1)$-th step.

\noindent
We can conclude that either there exists $\bar n$ such that $r_0 \ge (1-(\bar n+1)\gamma_0)(1-\bar n\gamma_0)$, and then $\gamma_{\bar n+1}>1$, or the sequence $\{ \gamma_n\}$ is given by $\gamma_n = \gamma_0/(1-n\gamma_0)$ for all $n\ge 1$, and since for all $\gamma_0\in (0,1)$ there exists $\tilde n$ such that $\gamma_0>1/\tilde n$, we have $\gamma_{\tilde n+1}>1$. Therefore
\[
n_1(\gamma,r) = 1+\min \left\{ \min\left\{ k\ge 1\, :\, r \ge (1-(k+1)\gamma)(1-k\gamma)\right\}\, ,\, \min \left\{ j\ge 1\, :\, \frac{\gamma}{1-j \gamma} \ge 1\right\} \right\}.
\]

\noindent
(ii) If $\gamma=1$, we are looking at the periodic point $(1,r,-1)$,
and we have argued that $\PPP_h^{\per(r)-1}(1,r,-1)=(1,r,+1)$ where
$\per(r)$ is defined in Theorem \ref{thm-periodic}. Hence
$n_2(1,r)=\per(r)-1$. For a general $\gamma>0$ we need to count the
intersections of the horocycle tangent to $\R$ at $\gamma$ and of
radius $r$, with the sets $\III_n$ with $n<0$ and with the geodesics
$\geod(p/q,p'/q')$ defined in \eqref{geod-pq}. The number of these
intersections give $n_2(\gamma,r)-1$.

The intersections with the lines $\III_n$ exist for all $n$ such that
$\gamma-r\le n<0$. In the formula for $n_2(\gamma,r)$ they correspond
to the cases $s=1$ and $m=|n|$, for which
\[
    M\left(\frac {|n|}{1} \right) = \begin{pmatrix} 1 & |n| \\ 0 &
        1 \end{pmatrix}
\]
and the condition is $\gamma+|n|\le r$. By Lemma
\ref{lem:inters-geod-horo}, the intersections with the geodesic
$\geod(p/q,p'/q')$ occur if $r\ge (q\gamma+p)(q'\gamma+p')$, and by
Lemmas \ref{lem:coding-geod-pq} and \ref{lem:via-matrici}, to such an
intersection it corresponds a point $(\gamma',r',\eps')\in \WW$ with
\[
\gamma'= \hat{M}\left(\frac pq\oplus \frac{p'}{q'}\right) (\gamma), \quad r' = \DDD{\hat{M}\left(\frac pq\oplus \frac{p'}{q'}\right)(\gamma)} r.
\]
Since in Section \ref{sec:trees} we have shown that for all $a,b\in \N$ with $(a,b)=1$ we have
\[
M\left(\frac ab \right) = M\left(\frac pq\oplus \frac{p'}{q'}\right) =  \begin{pmatrix} p' & p \\ q' & q \end{pmatrix} \quad \text{and} \quad \hat{M}\left(\frac pq\oplus \frac{p'}{q'}\right) =   \begin{pmatrix} q & p \\ q' & p' \end{pmatrix},
\]
we obtain the formula for $n_2(\gamma,r)$.
\\[0.15cm]
((iii) and (iv)) It is enough to repeatedly apply $\PPP_h$ as defined
in Tables \ref{tab-eps+1} and \ref{tab-eps0}.
\end{proof}

\begin{thm} \label{thm:acc-map}
Let $X:= ([1,\infty) \times [1,\infty) \times \{0, +1\}) \cap \WW$. For all $(\gamma,r,\eps) \in \WW$ there exists $\bar n(\gamma,r) \in \N$ such that $\PPP_h^{\bar n(\gamma,r)}(\gamma,r,\eps) \in X$. In particular it is well defined the map $\TTT_h:X \to X$ given by
\[
\TTT_h(\gamma, r, \eps) =  \PPP_h^{\tau(\gamma,r)}(\gamma, r, \eps)\, \quad \text{with} \quad \tau(\gamma,r) = \left\{ \begin{array}{ll} \lfloor \gamma \rfloor + 1 + n_2\left(\frac{1}{1+\lfloor \gamma \rfloor -\gamma}, \frac{r}{(1+\lfloor \gamma \rfloor -\gamma)^2}\right) & \text{if $\gamma \not\in \N$} \\[0.2cm] \gamma -1  & \text{if $\gamma \in \N$}
\end{array} \right.
\]
and $n_2(\gamma,r)$ defined as in Proposition \ref{prop-repar}. More specifically
\[
\TTT_h(\gamma, r, \eps) = \left( \frac{1}{1+\lfloor \gamma \rfloor - \gamma}\, ,\, \frac{r}{(1+\lfloor \gamma \rfloor - \gamma)^2}\, ,\, \tilde \eps \right)\, .
\]
\end{thm}

\begin{proof}
From Proposition \ref{prop-repar} we know that for all $(\gamma,r,-1)$ we end up at $(\gamma,r,+1)$ with $n_2(\gamma,r)$ iterations of $\PPP_h$, whereas if $\gamma\in (0,1)$ and $\eps\not= -1$ a point with $\gamma'\ge 1$ is reached with $n_1(\gamma,r)$ iterations. Together, these two steps imply the first part of the statement.

\noindent
It is then possible to define a map $\TTT_h$ from $X$ to itself. Let us check that the times $\tau(\gamma,r)$ may be chosen as in the statement. Considering first the case $\gamma \not\in \N$, Proposition \ref{prop-repar}-(iii) entails that the points $(\gamma,r,\eps)\in X$ hit $(\gamma-\lfloor \gamma \rfloor,r,+1)$ after $\lfloor \gamma \rfloor$ iterations of $\PPP_h$. Since $\gamma-\lfloor \gamma \rfloor \in (0,1)$ and $r\ge \gamma >1$, we can apply Proposition \ref{prop-repar}-(i) with $n_1(\gamma-\lfloor \gamma \rfloor,r)=1$, to get
\[
\PPP_h^{\lfloor \gamma \rfloor +1}(\gamma, r, \eps) =  \left( \frac{1}{1+\lfloor \gamma \rfloor - \gamma}\, ,\, \frac{r}{(1+\lfloor \gamma \rfloor - \gamma)^2}\, ,\, -1 \right)\, .
\]
We can now apply Proposition \ref{prop-repar}-(ii) and conclude the
computation of $\tau(\gamma,r)$ for $\gamma \not\in \N$. If
$\gamma \in \N$ the result immediately follows from Proposition
\ref{prop-repar}-(iv), with the standard convention that $\PPP_h^0 $
is the identity map.
\end{proof}

\begin{rem}
    We point out that the map $\TTT_h$ is not the first return map of
    $\PPP_h$ onto $X$. For $(\gamma,r,\eps)$ with $\gamma \in \N$,
    since we are considering points in the periodic orbit of
    $(1,r,+1)$, one can interpret $\TTT_h$ as the map that first sends
    $(\gamma,r,\eps)$ to the subset in $\WWW$ of points with
    $\gamma'\in (0,1]$, and then goes once round the periodic
    orbit. Hence it is the identity for $\gamma=1$. If
    $\gamma\not\in \N$, the interpretation may be similar. We first
    reach the subset in $\WWW$ of points with $\gamma'\in (0,1]$, and
    then ``follow'' a close periodic orbit, whose period only depends
    on $r$ and $\gamma$.
\end{rem}

Looking at the first component of the map $\TTT_h$, we are led to study the map
\[
T: [1,\infty) \to [1,\infty)\, , \quad T(t) = \frac{1}{1+\lfloor t \rfloor -t}
\]
for which an easy computation yields the following result.

\begin{prop}\label{prop-map-q}
The dynamical system $([1,\infty),T)$ is topologically conjugate to the backward continued fraction system $([0,1),F)$ given by
\[
F: [0,1) \to [0,1)\, ,\quad F(s) = \frac{1}{1-s}\pmod{1}
\]
with conjugacy $\phi : [1,\infty) \to [0,1)$ given by
$\phi(t) = 1-\frac 1t$. As a consequence the map $T$ preserves the
infinite measure $d\mu(t) = dt/(t(t-1))$. In particular, if $t\in \Q$
there exists $k_t\ge 0$ such that $T^{k_t}(t)=1$, and therefore
$T^k(t)=1$ for all $k\ge k_t$.
\end{prop}

\noindent
Let us now consider the action of $\TTT_h$ on the second
component. One notes that if $(\gamma,r,\eps)$ is not periodic as the
second component is strictly increasing since
$(1+\lfloor \gamma \rfloor -\gamma) <1$ for all $\gamma \not\in
\N$. Hence if we wait enough time, the map $\PPP_h$ makes the radius
of the horocycle to increase, thus converging to the cusp of $\MMM$.

\begin{rem} The map $F$ defined in Proposition \ref{prop-map-q} gets
    its name because can be obtained by flipping the standard
    continued fraction transformation about the vertical line
    $x=1/2$. Adler and Flatto showed in \cite{adler-flatto-2} that $F$
    is the map obtained by inducing on $(0,1)$ the first component of
    the Poincar\'e map on $\CCC$ of the geodesic flow $\tilde g_t$
    (see the map $\PPP_g$ constructed in Appendix \ref{app:geodesic},
    Equation (\ref{map-g})).  On the other hand, we showed above that
    $F$ is (conjugate to) the map on $[1,\infty)$ obtained by a
    time-reparameterization of the first component of the Poincar\'e
    map on $\CCC$ of the horocycle flow $\tilde h_s^+$.  We believe
    that this result helps to highlight how the parabolic horocycle
    flow also contains an ``expanding component'', which can be
    detected by an appropriate reparameterization.
\end{rem}

\section{Relations with \cite{ath-che}} \label{sec:athreya}

In \cite{ath-che} the authors have introduced a Poincar\'e section for the horocycle flow on $\MMM$ using the identification of $S\HP$ with the space of unimodular lattices. By adapting their argument to our notations, let
\[
\tilde \CCC := \set{(z,\zeta) \in S\MMM\, :\, y\ge 1\, ,\, 0< x\le 1\, ,\, \theta(\zeta)= \pi}\, ,
\]
that is all the vectors pointing downwards with base point in the
subset of the fundamental domain\footnote{Here we are thinking of the
  fundamental domain in the strip $x\in (0,1]$.} of $\MMM$ with
imaginary part greater than or equal to $1$. As in Section
\ref{sec:poinc-map} we set
\[
\tilde C := \set{(z,\zeta)\in S\HP\, :\, \pi_*(z,\zeta) \in \tilde \CCC}\, .
\]
Setting
\[
\Omega:= \set{(\alpha,\beta)\in \R^2\, :\, \alpha,\beta\in (0,1]\, ,\, \alpha+\beta>1}\, ,
\]
the set $\tilde \CCC$ may be identified with $\Omega$ by
\begin{equation}\label{the-point-z}
z= \left\{ \begin{array}{ll}
\frac \beta \alpha - \left\lfloor \frac \beta \alpha \right\rfloor + i \frac 1{\alpha^2} & \text{if $\frac \beta \alpha \not\in \Z$}\\[0.2cm]
1+ i \frac 1{\alpha^2} & \text{if $\frac \beta \alpha \in \N$}
\end{array} \right..
\end{equation}

\begin{thm}[\cite{ath-che}]
The Poincar\'e map of the horocycle flow $\tilde h_s^+$ with negative $s$ on the section $\tilde \CCC$ in the variables $(\alpha,\beta)\in \Omega$ is the Boca-Cobeli-Zaharescu map studied in \cite{BCZ} and defined to be
\[
V: \Omega \to \Omega\, ,\quad V(\alpha,\beta) = \left(\beta, -\alpha +\left\lfloor \frac{1+\alpha}{\beta} \right\rfloor \beta \right)\, .
\]
\end{thm}

\noindent
One of the main features of the map $V$ is its relation with the Farey
sequences. We recall that for each fixed $Q\in \N$ the \emph{Farey
  sequence} $\FFF(Q)$ is the collection of fractions in $[0,1]$ which
written in reduced form have denominator less than or equal to $Q$. If
\[
\FFF(Q) = \set{ \frac 01 = \frac{p_0}{q_0} < \frac 1Q =  \frac{p_1}{q_1} <  \frac{p_2}{q_2} < \dots <  \frac{p_{N-1}}{q_{N-1}} < \frac 11 =  \frac{p_N}{q_N}},
\]
the relation between $V$ and $\FFF(Q)$ introduced in \cite{BCZ} is that
\[
V\left(\frac{q_{i-1}}{Q},\frac{q_{i}}{Q}\right) = \left(\frac{q_{i}}{Q},\frac{q_{i+1}}{Q}\right)
\]
for all $i=1,\,\dots,\,N$, with the convention $q_{N+1}=q_0$. Hence the denominators of the fractions in the Farey sequence $\FFF(Q)$ describe a periodic orbit of $V$ of period $N=N(Q)$, one less than the cardinality of $\FFF(Q)$. Moreover, it is shown in \cite{ath-che} that this periodic orbit corresponds to a closed horocycle of length $Q^2$.

\noindent
Let us find the corresponding periodic orbit for the Poincar\'e map $\PPP_h$. As explained in Remark \ref{rem-closed-w}, the closed horocycle of length $Q^2$ corresponds to the $\PPP_h$-periodic point $(1, Q^2/2, -1) \in \WW$ for $Q\ge 2$, with orbit
\[
\OOO\left(1,\frac{Q^2}{2},-1\right) = \set{\left(\frac ab, \frac{Q^2}{2\, b^2}, \pm 1\right) \, :\, (a,b)=1\, ,\, ab<\frac{Q^2}{2}} \cup \set{\left(\frac ab, \frac{Q^2}{2\, b^2}, 0\right) \, :\, (a,b)=1\, ,\, ab=\frac{Q^2}{2}}\, .
\]
Therefore if one follows the closed horocycle $W^+_{Q^2/2}$, tangent to $\R$ at 1 and of radius $Q^2/2$, from its lowest point on $\III_0$ until it closes on $S\MMM$, then one finds $N(Q)$ passages through the section $\tilde \CCC$. Each of these passages occurs between two consecutive returns to the section $\CCC$ we have used to define the map $\PPP_h$. In particular, we use the suspension flow $\phi_s$ defined in \eqref{susp-flow-horo} in terms of the map $\PPP_h$ to identify the passages through $\tilde \CCC$.

\begin{prop} \label{prop:rel-ath-che}
For all $Q\ge 2$, for a point $(\gamma, r, \eps) \in \OOO(1,Q^2/2,-1)$ there exists $s_* \in (s_h(\gamma,r,\eps),0)$ such that $\phi_{s_*}(\gamma,r,\eps,0) \in \tilde \CCC$ if and only if $\gamma = a/b$ and $r=Q^2/(2b^2)$, with $(a,b)\in \N\times \N$ in the set
\[
\tilde \FFF(Q) := \set{(a,b)=1,\, a\le b\le Q,\, ab\le \frac{Q^2}{2}} \cup \set{(a,b)=1, \, a>b,\, a+b\le Q,\, a(a+b)> \frac{Q^2}{2}}.
\]
Hence $\#\,  \tilde \FFF(Q) = N(Q)$ for all $Q\ge 2$. In addition, the point $\phi_{s_*}(a/b, Q^2/(2b^2), \eps,0)$ corresponds to the point $\left(\alpha(a,b), \beta(a,b)\right) \in \Omega$ given by
\[
\left(\alpha(a,b), \beta(a,b)\right) =\left\{ \begin{array}{ll} \left(\frac bQ\, ,\, \frac{a+\lfloor \frac{Q-a}b\rfloor b}Q\right)  & \text{if $\frac ab \le 1$} \\[0.3cm] \left( \frac{a+b}{Q}\, ,\, \frac{a}{Q}\right) & \text{if $\frac ab > 1$}
\end{array} \right..
\]
\end{prop}

\begin{proof}
Let $Q\ge 2$ and consider all points $(\gamma, r, \eps) \in \OOO(1,Q^2/2,-1)$. Looking at the part of the horocycle flow from the point in $\CCC$ corresponding to $(\gamma, r, \eps)$ to its first return to $\CCC$, we need to determine whether there is a passage through $\tilde \CCC$.

\noindent
Let us first consider the case $\eps=+1$. If $\gamma = a/b > 1$ then, by Table \ref{tab-eps+1}, the first return to $\CCC$ occurs on $\III_1$ with $\PPP_h(a/b, Q^2/(2b^2),+1) = ((a-b)/b,Q^2/(2b^2),+1)$. Therefore along this portion of the horocycle flow there is no passage through $\tilde \CCC$, since the vector never points downward. If $\gamma=a/b=1$, then $a=b=1$ and the first return to $\CCC$ occurs on $\III_2$. In this case, before the return to $\III_2$, the horocycle $W^+_{Q^2/2}$ intersects $\III_1$ and on this intersection the vector points downward. The intersection point is $z = 1 + i Q^2$, so that from \eqref{the-point-z}
\[
\alpha(1,1) = \frac 1Q
\]
and $\beta/\alpha \in \N$, which implies
\[
\beta(1,1) = 1 = \frac QQ\, .
\]
If $\gamma = a/b <1$ and $b\le Q$, then the return to $\CCC$ occurs after a passage through the point
\[
z = \frac ab + i \frac{Q^2}{b^2}
\]
with vector pointing downward. This yields a passage through $\tilde \CCC$ with
\[
\alpha(a,b) = \frac bQ, \quad \frac{\beta(a,b)}{\alpha(a,b)} - \left\lfloor \frac{\beta(a,b)}{\alpha(a,b)} \right\rfloor = \frac ab
\]
by \eqref{the-point-z}, from which
\[
\beta(a,b) = \frac{a+\lfloor \frac{Q-a}b\rfloor b}Q\, .
\]
If instead $b>Q$, then there is no passage through $\tilde \CCC$ before the first return to $\CCC$. Otherwise, there would be a point $(\bar z, \bar \zeta) \in S\MMM$ with $\bar z\in \III^+$, that is with imaginary part greater or equal than 1, and $\theta(\bar\zeta) \in (-\pi, 0)$, equivalent in $S\MMM$ to $(z,\zeta)=\WWW^{-1}(a/b,Q^2/(2b^2),+1)$. But this is impossible, since $S_*(z,\zeta)$ has base point in $\III^+$ and angle in $(0,\pi)$.

\noindent
For the case $\eps=0$ we can repeat exactly the same argument. We have thus proved half of the result. Indeed, we have proved that for $\eps=0,+1$, for a point $(\gamma, r, \eps) \in \OOO(1,Q^2/2,-1)$ there exists $s_* \in (s_h(\gamma,r,\eps),0)$ such that $\phi_{s_*}(\gamma,r,\eps,0) \in \tilde \CCC$ if and only if $\gamma = a/b$ and $r=Q^2/(2b^2)$, with $(a,b)\in \N\times \N$ in the set
\[
\set{(a,b)=1,\, a\le b\le Q,\, ab\le \frac{Q^2}{2}}.
\]
In addition the point $\phi_{s_*}(a/b, Q^2/(2b^2), \eps,0)$ corresponds to the point
\[
\left(\alpha(a,b), \beta(a,b)\right) = \left(\frac bQ\, ,\, \frac{a+\lfloor \frac{Q-a}b\rfloor b}Q\right) \in \Omega.
\]

\noindent
We now consider the case $\eps=-1$. Looking at Table \ref{tab-eps-1}, if the radius is big enough the first return to $\CCC$ occurs on $\JJJ_{-1}$. Hence we are dealing with the portion of the flow on the horocycle $W^+$ tangent to $\R$ at a point $\gamma$ and of radius $r$ between $\III_0$ and $\JJJ_{-1}$. Let us consider the action of the matrix $S\in SL(2,\Z)$ on this horocycle. Since $S(\III_0)=\III_0$ and $S(\JJJ_{-1})=\III_1$, we get  the flow on $S(W^+)$ between $\III_0$ and $\III_1$. In addition, if $(z,\zeta)= \WWW^{-1}(\gamma,r,-1)$ (see \eqref{eq-map-W}), then $(\bar z, \bar \zeta) = S_*(z,\zeta)$ satisfies $\theta(\bar \zeta)\in (\pi/2, \pi)$ since $\theta(\zeta) \in (-\pi/2, 0)$. Therefore along the portion of the flow on $S(W^+)$ we are looking at, there is no passage through $\tilde \CCC$ because the vector never points downward. Since $W^+$ and $S(W^+)$ are equivalent in $S\MMM$, this shows that in this first case there is no passage through $\tilde \CCC$ before the first return to $\CCC$.

\noindent
By Table \ref{tab-eps-1} we are thus left with two cases to consider. The first corresponds to $\gamma=a/b >1$, $r = Q^2/(2b^2) <\gamma(1+\gamma)= a(a+b)/b^2$, and first return to $\CCC$ on $\III_{-1}$. We argue as before by looking at the action of the matrix $L\in SL(2,\Z)$. If $W^+$ is the horocycle tangent to $\R$ at $\gamma$ and of radius $r$, then $L(W^+)$ is the horocycle tangent to $\R$ at $L(\gamma)=\gamma/(1+\gamma) \in (0,1)$ and of radius $r/(1+\gamma)^2 < L(\gamma)$. This new horocycle does not intersect $\III_0$, and since $L(\III_0)=\JJJ_0$ and $L(\III_{-1})=\III_1$, we look at the portion of the flow along this horocycle from $\JJJ_0$ to $\III_1$. Since $L(\gamma)\in (0,1)$, there exists a point
\[
z= \frac{\gamma}{1+\gamma} + i \frac{2r}{(1+\gamma)^2} = \frac{a}{a+b} + i \frac{Q^2}{(a+b)^2}
\]
at which the vector points downward. Therefore, if $Q^2/(a+b)^2 \ge 1$, we have found a passage through $\tilde \CCC$. If instead $Q^2/(a+b)^2 < 1$ there is no passage through $\tilde \CCC$ as can be shown by arguing as in the case $\eps=+1$, $\gamma>1$ and $b>Q$ before.
Therefore, in this first case we obtain a passage through $\tilde \CCC$ if and only if $(a,b)$ is in the set
\[
\set{(a,b)=1, \, a>b,\, a+b\le Q,\, ab\le \frac{Q^2}{2}< a(a+b)}
\]
and one realises that the condition $(a+b)\le Q$ implies $ab\le Q^2/2$. In addition, using \eqref{the-point-z}, we obtain
\[
\alpha(a,b) = \frac{a+b}{Q}\, , \quad \frac{\beta(a,b)}{\alpha(a,b)} - \left\lfloor \frac{\beta(a,b)}{\alpha(a,b)} \right\rfloor = \frac aQ\, .
\]
It follows that $a\ge Q/2$, hence $2a+b>Q$, so that $\beta(a,b) = a/Q$.

\noindent
The second case we need to consider corresponds to $r  < 1 + \gamma$ and the first return to $\CCC$ occuring on $\III_0$. In this case, the horocycle $W^+$ along which we are flowing is equivalent to $L(W^+)$, which needs to intersect $\JJJ_0$ in two points, without intersecting other copies of $\III_0$ in the meantime. Again we are using that $L(\III_0)=\JJJ_0$. It means that the radius $r/(1+\gamma)^2$ of $L(W^+)$ is smaller than $1/2$. Hence no point along the horocycle $L(W^+)$ is in $\tilde \CCC$.

\noindent
We have thus examined all the points in the orbit of $(1,Q^2/2,-1)$, finding all the passages through $\tilde \CCC$. The proof is finished.
\end{proof}

\noindent
The ideas in the proof of Proposition \ref{prop:rel-ath-che} may be used also for non-periodic orbits. One can show that following the orbit of a point $(\gamma_0,r_0,\eps_0)\in \WW$, there is a passage through the section $\tilde \CCC$ between two consecutive returns to $\CCC$, $(\gamma_n, r_n, \eps_n) = \PPP_h^n(\gamma,r,\eps)$ and $(\gamma_{n+1}, r_{n+1}, \eps_{n+1}) = \PPP_h^{n+1}(\gamma,r,\eps)$ for $n\ge 0$, if and only if
\[
(\gamma_n, r_n, \eps_n) \in \set{\gamma\le 1\, ,\, r\ge \frac 12\, ,\, \eps\not= -1}\cup \set{\gamma>1\, ,\, \frac{(1+\gamma)^2}{2}\le r<\gamma(1+\gamma)\, ,\,  \eps=-1}\, .
\]

\appendix

\section{Proof of Lemma \ref{lem:matr-perm}} \label{app:proof}

We check that the relation holds for different combinations of the
matrices $L$ and $R$ in $M(p/q)$. Let us assume that $M(p/q)$ starts
with $L$ and ends with $R$, hence it can be written as
\begin{equation} \label{first-case-perm}
M\left( \frac pq \right) = L^{n_1}R^{m_1}L^{n_2}R^{m_2}\dots L^{n_s}R^{m_s}
\end{equation}
for $s\ge 1$, $n_1\ge 1$, and $m_s\ge 1$. We argue by induction on the number of blocks $L^{n_j}R^{m_j}$ in the matrix. Recall that
\[
L^n = \begin{pmatrix} 1 & 0 \\ n & 1 \end{pmatrix}, \qquad R^m = \begin{pmatrix} 1 & m \\ 0 & 1 \end{pmatrix}.
\]
If $s=1$ it is enough to compute
\[
L^nR^m = \begin{pmatrix} 1 & m \\ n & nm+1 \end{pmatrix} \quad \text{and} \quad R^mL^n = \begin{pmatrix} nm+1 & m \\ n & 1 \end{pmatrix}.
\]
Assuming that the statement holds for $s-1\ge 1$, let's prove it holds for $s$. Writing $M(p/q)$ as
\[
M\left( \frac pq \right) = M'\, L^{n_s}R^{m_s}, \quad \text{with} \quad M' = L^{n_1}R^{m_1}L^{n_2}R^{m_2}\dots L^{n_{s-1}}R^{m_{s-1}},
\]
we have
\[
\hat{M}\left( \frac pq \right) = R^{m_s} L^{n_s} \hat{M}'
\]
being $\hat{M}'$ the reverse function of $M'$. Since we can apply the statement to $M'$ by the inductive assumption,
\[
M' = \begin{pmatrix} a' & b' \\ c' & d'
\end{pmatrix} \quad \Leftrightarrow \quad \hat{M}' = \begin{pmatrix} d' & b' \\ c' & a'
\end{pmatrix},
\]
and we can compute that
\[
M\left( \frac pq \right) = \begin{pmatrix} a'+b'n_s & a'm_s + b'(n_sm_s+1) \\ c'+d'n_s & c'm_s+d'(n_sm_s+1)
\end{pmatrix}
\]
and
\[
\hat{M}\left( \frac pq \right) = \begin{pmatrix} d'(n_sm_s+1)+c'm_s & b'(n_sm_s+1)+a'm_s \\ d'n_s+c' & b'n_s+a'
\end{pmatrix}.
\]
Thus the lemma is proved if $M(p/q)$ is written as in
\eqref{first-case-perm}. The second case we consider is when $M(p/q)$
starts with $R$ and ends with $L$. However this case can be reduced to
the first one by exchanging the role of $M(p/q)$ and its
reverse. Lastly we consider the cases when $M(p/q)$ starts and ends
with the same matrix, either $L$ or $R$. Let us assume that $M(p/q)$
can be written as
\begin{equation} \label{last-case-perm}
M\left( \frac pq \right) = L^{n_1}R^{m_1}L^{n_2}R^{m_2}\dots L^{n_{s-1}}R^{m_{s-1}}L^{n_s}
\end{equation}
for $s\ge 1$, $n_1\ge 1$, and $n_s\ge 1$. If $s=1$ the matrix contains only $L$ and coincide with its reverse. The statement follows by the form of the matrices $L^n$. If $s>1$ we can write
\[
M\left( \frac pq \right) = M' L^{n_s}
\]
with
\[
M'= L^{n_1}R^{m_1}L^{n_2}R^{m_2}\dots L^{n_{s-1}}R^{m_{s-1}}
\]
a matrix of the case considered in \eqref{first-case-perm}. Hence the statement is true for $M'$ and
\[
M' = \begin{pmatrix} a' & b' \\ c' & d'
\end{pmatrix} \quad \Leftrightarrow \quad \hat{M}' = \begin{pmatrix} d' & b' \\ c' & a'
\end{pmatrix},
\]
Hence we can compute that
\[
M\left( \frac pq \right) = \begin{pmatrix} a'+b'n_s & b' \\ c'+d'n_s & d'
\end{pmatrix} \quad \text{and} \quad \hat{M}\left( \frac pq \right) = \begin{pmatrix} d' & b' \\ d'n_s+c' & b'n_s+a'
\end{pmatrix}.
\]
Thus the lemma is proved if $M(p/q)$ is written as in \eqref{last-case-perm}. The analogous argument works when $M(p/q)$ starts and ends with $R$.

\section{The Poincar\'e map for the geodesic flow and its suspension} \label{app:geodesic}

In this appendix we show for completeness the construction of the
Poincar\'e map for the geodesic flow with respect to the Poincar\'e
section $\CCC$ defined in \eqref{not-poin-sec}. This is similar in
spirit to the results in \cite{adler-flatto,series} and uses
essentially the same Poincar\'e section as in \cite{mayer-note}.

\noindent
Let $(z,\zeta)\in \CCC$ and $\gamma_{z,\zeta}(t)$ be the
geodesic tangent to $\zeta$ at $z$. Since
$\theta(\zeta) \in (-\pi,0)$ we have $\gamma_{z,\zeta}^+\in \R^+$ and
$\gamma_{z,\zeta}^-\in \R^-$. We let
\[
    \GG :=\set{(\gamma,\eta) \,:\, \gamma,\eta\in \R^+}
\]
and consider the map
\[
\GGG : \CCC \to \GG, \qquad (z,\zeta) \mapsto \GGG(z,\zeta) = \left(\gamma_{z,\zeta}^+, -\gamma_{z,\zeta}^-\right).
\]
Using the identification of $S\HP$ with $PSL(2,\R)$, the map $\GGG$ can be defined also as a map from $SL(2,\R)$ to $\GG$. Following the notation of Section \ref{sec:facts}, we have
\begin{equation}\label{map-sl2r-gg-inv}
(z,\zeta) = \GGG^{-1}(\gamma, \eta) \quad \Rightarrow \quad \Gamma_{(z,\zeta)} = \begin{pmatrix} \left(\frac \eta \gamma\right)^{\frac 14}\, \frac{\gamma}{\sqrt{\gamma+\eta}} & -\left(\frac \gamma \eta\right)^{\frac 14}\, \frac{\eta}{\sqrt{\gamma+\eta}}\\[0.2cm] \left(\frac \eta \gamma\right)^{\frac 14}\, \frac{1}{\sqrt{\gamma+\eta}} & \left(\frac \gamma \eta\right)^{\frac 14}\, \frac{1}{\sqrt{\gamma+\eta}} \end{pmatrix} \, .
\end{equation}
We now consider the Poincar\'e map on $\CCC$ of the geodesic flow $\tilde g_t$ starting from a point $(z,\zeta)\in \CCC$ and along $\gamma_{z,\zeta}(t)$ for positive times $t$. Let $(\gamma,\eta) = \GGG(z,\zeta)$. Given $(z',\zeta')$ the point in $\CCC$ of the first return, we can describe the Poincar\'e map as a map
\[
\PPP_g : \GG \to \GG, \quad \PPP_g(\gamma,\eta) = (\gamma',\eta'):= \GGG(z',\zeta')\, .
\]
Our aim is now to describe the map $\PPP_g$ on the points of $\GG$, also in relation with the $\{L,R\}$ coding of $\gamma \in \R^+$ introduced in Section \ref{sec:trees}. The map $\PPP_g$ depends on the set of $C$ defined in \eqref{equiv-poin-sec} on which the first return occurs. The following analogue of Lemma \ref{lem:via-matrici} holds and the proof is similar.

\begin{lem}\label{lem:via-matrici-g}
For $(z,\zeta)\in \CCC$, let $(\gamma,\eta) = \GGG(z,\zeta)$. If the first return to $C$ for $\tilde g_t$ with positive $t$ along $\gamma_{z,\zeta}(t)$ occurs on $\Gamma^{-1}(\III_0)$ for some $\Gamma\in PSL(2,\Z)$, then $(\gamma',\eta')=\PPP_g(\gamma,\eta)$ satisfies
\[
\gamma' = \Gamma(\gamma)\quad \text{and} \quad \eta' = -\Gamma(-\eta).
\]
\end{lem}

\noindent
Different cases have to be considered (see Figure
\ref{fig-geodflow}). Given $(z,\zeta)=\GGG^{-1}(\gamma,\eta)$, if
$\gamma>1$ the first return to $C$ for $g_t$ with $t$ positive along
$\gamma_{z,\zeta}(t)$ occurs on $\III_1=R(\III_0)$. On the other hand,
if $\gamma<1$ it occurs on $\JJJ_0=L(\III_0)$. Finally, when
$\gamma=1$ there are no returns to $C$ along $\gamma_{z,\zeta}(t)$ for
positive $t$. However we can extend the map also to this case, and
using Lemma \ref{lem:via-matrici-g} we obtain
\begin{equation}\label{map-g}
    (\gamma',\eta') = \PPP_g(\gamma,\eta) =
    \begin{cases}
        (\gamma -1, \eta+1) & \text{if $\gamma>1$}
        \\[0.2cm]
        \left( \frac{\gamma}{1- \gamma},\,\frac{\eta}{1+ \eta} \right)
        & \text{if $\gamma<1$}
        \\[0.2cm]
        \left( 0, \infty \right) & \text{if $\gamma=1$}
    \end{cases}
\end{equation}
so that we recognise the map $U$ defined in \eqref{ext-farey-map} as the factor map of $\PPP_g$ on the first coordinate.

\begin{figure}[h!]
    \pgfmathsetmacro{\myxlow}{-1.5}
\pgfmathsetmacro{\myxhigh}{1.5}
\pgfmathsetmacro{\myiterations}{2}

\begin{tikzpicture}[scale=3.5]

    \draw[-latex',thin](\myxlow-0.1,0) -- (\myxhigh+0.2,0);
    \draw[-latex',thin,name path=I0](0,0) -- (0,\myxhigh);

    \pgfmathsetmacro{\succofmyxlow}{\myxlow+1}
    \foreach \x in {-1,-0.5,0,0.5,1}
    {
      \draw (\x,0) -- (\x,-0.05);
    }
    \foreach \x in {-1,0,1}
    {
      \draw (\x,0) -- (\x,-0.05) node[below]{$\x$};
    }
    \draw (-0.5,0) -- (-0.5,-0.05) node[below]{$-\frac 12$};
    \draw (0.5,0) -- (0.5,-0.05) node[below]{$\frac 12$};

    \draw[dashed,farey,name path=I-1] (-1,0) -- (-1,\myxhigh) node[above]{$\III_{-1}$};
    \draw[dashed,farey] (0,0) -- (0,\myxhigh) node[above]{$\III_0$};
    \draw[dashed,farey,name path=I1] (1,0) -- (1,\myxhigh) node[above]{$\III_1$};
    \draw[thin,farey,name path=J-1] (0,0) arc(0:180:0.5);
    \draw[thin,farey,name path=J0] (1,0) arc(0:180:0.5);
    \node [farey] at (160:1.1) {$\JJJ_{-1}$};
    \node [farey] at (30:1) {$\JJJ_0$};

    \def\et{0.35}
    \def\g{0.8}
    \def\xc{(\g-\et)/2}

    \draw[thick,blue,name path=G1] (-\et,0) arc(180:0:{(\g+\et)/2});
    \node[blue] at (\g,0) {$\bullet$};
    \node[blue] at (-\et,0) {$\bullet$};
    \node[blue,below] at (\g,0) {$\gamma$};
    \node[blue,below] at (-\et,0) {$-\eta$};

    \path [name intersections={of=I0 and G1,by=z}];
    \path [name intersections={of=J0 and G1,by=z'}];
    \node [blue,above left] at (z) {$(z,\zeta)$};
    \draw[blue,-triangle 45,rotate around={90:(z)}] (z) -- ($(z)!0.3cm!({\xc},0)$);
    \node [blue,above] at (z') {$(z',\zeta')$};
    \draw[blue,-triangle 45,rotate around={90:(z')}] (z') -- ($(z')!0.3cm!({\xc},0)$);

    \def\et{0.7}
    \def\g{1.3}
    \def\xc{(\g-\et)/2}

    \draw[thick,Red,name path=G1] (-\et,0) arc(180:0:{(\g+\et)/2});
    \node[Red] at (\g,0) {$\bullet$};
    \node[Red] at (-\et,0) {$\bullet$};
    \node[Red,below] at (\g,0) {$\gamma$};
    \node[Red,below] at (-\et,0) {$-\eta$};

    \path [name intersections={of=I0 and G1,by=z}];
    \path [name intersections={of=I1 and G1,by=z'}];
    \node [Red,above left] at (z) {$(z,\zeta)$};
    \draw[Red,-triangle 45,rotate around={90:(z)}] (z) -- ($(z)!0.3cm!({\xc},0)$);
    \node [Red,right] at (z') {$(z',\zeta')$};
    \draw[Red,-triangle 45,rotate around={90:(z')}] (z') -- ($(z')!0.3cm!({\xc},0)$);

\end{tikzpicture}
    \caption{Action of the first return map $\PPP_g$. The red geodesic
      has $\gamma>1$ and thus the first return to $C$ of the geodesic
      flow occurs on $\III_1$; the blue geodesic has $\gamma<1$, so
      that the first return to $C$ of the geodesic flow occurs on
      $\JJJ_0$.}\label{fig-geodflow}
\end{figure}
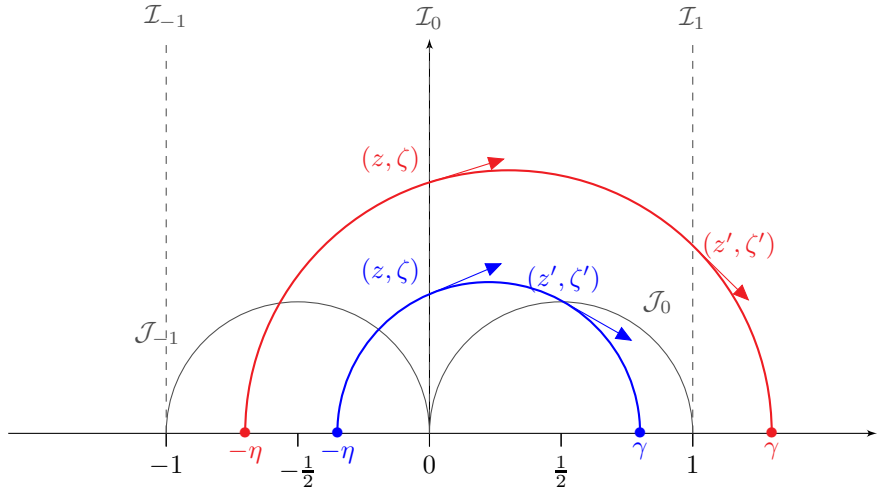

\noindent
Let us now consider the $\{L,R\}$ coding of $\gamma \in \R^+$. Using that $M(\gamma) = R\cdots$ when $\gamma>1$, $M(\gamma) = L\cdots$ when $\gamma<1$, and $M(\gamma) = I$ when $\gamma=1$, we obtain
\[
    M(\gamma') =
    \begin{cases}
        \sigma(M(\gamma)) & \text{if $\gamma\not=1$} \\
        \lambda & \text{if $\gamma=1$}
    \end{cases},
\]
where $\lambda$ denotes the empty word.

\noindent
We can now construct the suspension flow associated to
$(\GG,\PPP_g)$. Let $(z,\zeta)\in \CCC$ and apply the geodesic flow
along $\gamma_{z,\zeta}(t)$ for $t$ positive. Letting
$(\gamma,\eta) := \GGG(z,\zeta)$, the first return of the geodesic
flow $g_t(z,\zeta)$ on $C$ occurs at a time $t_g(z,\zeta)>0$, which
can then be written as a function of $(\gamma, \eta)$ as
$\rho_g(\gamma,\eta) := t_g(\GGG^{-1}(\gamma,\eta))$. Let us now
compute $\rho_g(\gamma,\eta)$. When $\gamma>1$, we have seen that the
first return on $C$ along $\gamma_{z,\zeta}(t)$ for $t$ positive
occurs on $\III_1$. Hence, using the identification between $S\HP$ and
$PSL(2,\R)$ and \eqref{geodesic-sl2r}, we write
\[
\Gamma_{g_t(z,\zeta)} = \begin{pmatrix} a & b\\ c & d \end{pmatrix} \, \begin{pmatrix} e^{\frac t2} & 0 \\ 0 & e^{-\frac t2} \end{pmatrix} \quad \text{where}\, \,  \Gamma_{(z,\zeta)} = \begin{pmatrix} a & b\\ c & d \end{pmatrix},
\]
from which it follows that $t_g(z,\zeta)$ is the positive solution of
\[
    \Real \left( \frac{a\, e^{\frac t2} i + b\, e^{-\frac t2}}{c\, e^{\frac t2} i + d\, e^{-\frac t2}}  \right) = 1\, .
\]
A straightforward computation and \eqref{map-sl2r-gg-inv} give
\[
t_g(z,\zeta) = \frac 12 \, \log \left( \frac{ac+d^2}{ac-c^2} \right) \quad \text{and} \quad \rho_g(\gamma,\eta) = \frac 12 \, \log \left( \frac{1+\frac 1 \eta}{1-\frac 1 \gamma} \right) \, \, \text{if $\gamma>1$.}
\]
When $\gamma<1$, we have seen that the first return on $C$ along
$\gamma_{z,\zeta}(t)$ for $t$ positive occurs on $\JJJ_0$. Hence, as
before it follows that $t_g(z,\zeta)$ is the positive solution of
\[
\left| \left( \frac{a\, e^{\frac t2} i + b\, e^{-\frac t2}}{c\, e^{\frac t2} i + d\, e^{-\frac t2}}  \right) - \frac 12\right|= \frac 12\, .
\]
A straightforward computation and \eqref{map-sl2r-gg-inv} give
\[
t_g(z,\zeta) = \frac 12 \, \log \left( \frac{b^2-bd}{ac-a^2} \right) \quad \text{and} \quad \rho_g(\gamma,\eta) = \frac 12 \, \log \left( \frac{1+\eta}{1-\gamma} \right) \, \, \text{if $\gamma<1$.}
\]
Let now
\[
    \Sigma_\rho := \set{(\gamma,\eta,\tau) \in \R^+\times \R^+\times
      \R\, :\, \gamma \not=1\, ,\, 0\le \tau \le \rho_g(\gamma,\eta)}
\]
and consider the flow $\phi_t : \Sigma_\rho \to \Sigma_\rho$ defined by
\[
    \phi_t (\gamma, \eta, \tau) =
    \begin{cases}
        (\gamma, \eta, \tau+t) & \text{if
          $0\le \tau+t<\rho_g(\gamma,\eta)$}
        \\[0.1cm]
        (\PPP_g(\gamma, \eta), 0) & \text{if
          $\tau+t=\rho_g(\gamma,\eta)$}
    \end{cases},
\]
and analogously for other values of $t$. Then $(\Sigma_\rho,\phi_t)$
is the suspension flow over the map $\PPP_g$, for which the following
proposition holds.

\begin{prop}\label{mis-inv-geodesic}
    The following properties hold.
    \begin{enumerate}
      \item The map $\PPP_g :\GG \to \GG$ preserves the infinite
        measure $\mu$ which is absolutely continuous with respect to
        the Lebesgue measure on $\GG$ with density
        $h(\gamma,\eta)=(\gamma+\eta)^{-2}$.
      \item The flow $\phi_t : \Sigma_\rho \to \Sigma_\rho$ preserves
        the measure $\tilde \mu$ which is absolutely continuous with
        respect to the Lebesgue measure on $\Sigma_\rho$ with density
        $\tilde h(\gamma,\eta,\tau)=(\gamma+\eta)^{-2}$. Moreover
        $\tilde \mu(\Sigma_\rho) = \pi^2/3$.
      \item The dynamical system $(\Sigma_\rho, \tilde \mu, \phi_t)$
        is isomorphic to the system $(S\MMM,\tilde m, \tilde g_t)$,
        where $\tilde m$ is the projection on $S\MMM$ of the Liouville
        measure $dm(x,y,\theta) = y^{-2}\, dx\, dy\, d\theta$ on
        $S\HP$, where $z=x+i y$ and $\theta = \theta(\zeta)$.
    \end{enumerate}
\end{prop}

\begin{proof}
Given $(z,\zeta)\in S\HP$ with $\theta(\zeta)\not\in \set{0,\pi}$, consider its coordinates $(\gamma,\eta,\tau)$ defined by $(\gamma,\eta) = \GGG(z,\zeta)$, and $\tau\in \R$ such that $(z,\zeta) = g_\tau(z',\zeta')$ where $z'$ is the point of the geodesic $\gamma_{z,\zeta}(t)$ with maximum value of the $y$ component. We obtain the following expressions for the coordinates $(x,y,\theta)$ of $(z,\zeta)$, where $z=x+i y$ and $\theta=\theta(\zeta)$:
\begin{equation}\label{coord-ch-xytheta}
x(\gamma,\eta,\tau) = \frac{\gamma\, e^\tau - \eta\, e^{-\tau}}{e^\tau + e^{-\tau}}\, ,\quad y(\gamma,\eta,\tau) = \frac{\gamma+ \eta}{e^\tau + e^{-\tau}}\, ,\quad \theta(\gamma,\eta,\tau) = -2\, \arctan (e^\tau)\, .
\end{equation}
From \eqref{coord-ch-xytheta} it follows that the Liouville measure $dm(x,y,\theta) = y^{-2}\, dx\, dy\, d\theta$ becomes
\[
2\, d\tilde \mu = \frac{2}{(\gamma+\eta)^2}\, d\gamma\, d\eta\, d\tau
\]
in the coordinates $(\gamma,\eta,\tau)$. Moreover by the definition of $\tau$, the geodesic flow acts in the coordinates $(\gamma,\eta,\tau)$ simply by translation on $\tau$, that is
\[
g_t(\gamma,\eta,\tau) = (\gamma,\eta,\tau+t)\, .
\]
These facts, together with $\tilde m(S\MMM) = 2\pi^2/3$, prove (ii)
and (iii). Finally (i) follows from (ii).
\end{proof}

\section{The dynamics of the periodic points of $\PPP_h$ on the permuted Stern-Brocot tree} \label{app:algorithm}

Here we describe the algorithm to find the points of the periodic
orbit of $(1,r,-1)$ for $\PPP_h$ in their dynamical order up to
$(1,r,+1)$ (refer to Figure \ref{fig-motions} where the case
$r\in (10,11)$ is shown):
\begin{enumerate}
  \item[(a)] the initial condition is the point $(1,r,-1)\in \WW$ and for completeness we consider the case $r\in \N$, with $r\ge 2$ (the case $r=1$ has already been studied and for $r\in [1,2)$ no fraction other than $1/1$ in $\hat{\TTT}$ has energy small enough to be in the periodic orbit). If $r\not\in \N$ then the points with $\eps=0$ should be excluded in the next steps. Then we start with the fraction $1/1$;
\item[(b)] go down on the left along $\hat{\TTT}$ finding the fractions $1/n$ for $n\le r$ (here we are applying $L$ to the $\gamma$-components). These fractions correspond to the points $(1/n, r/n^2, -1)\in \WW$ if $2\le n<r$, and to the point $(1/n, 1/n, 0)\in \WW$ if $n=r$. Note that for the energy $\xi(1/r)=r$, and it is an elementary remark that there is only another fraction at the same level in $\hat{\TTT}$ with energy greater than or equal to $r$, namely $r/1$;
\item[(c)] if $r=2$, it is clear that the only other fraction in $\hat{\TTT}$ with energy less than or equal to $r=2$ is $2/1$, hence we obtain the point $(2/1,2,0)$, from which we get $(1,r,+1)$;
\item[(d)] if $r\ge 3$ go up back to the fraction $1/(r-1)$ (applying $L^{-1}$), to find the corresponding point $(1/(r-1), r/(r-1)^2, +1)\in \WW$. The right sister of $1/(r-1)$ is the fraction $(r-1)/(r-2)$, and $\xi((r-1)/(r-2))>r$ for all $r\ge 4$. If $r=3$, the right sister of $1/2$ is $2/1$, we get the point $(2/1,3,-1)\in \WW$, and we continue as in step (i);
\item[(e)] if $r\ge 4$, keep going up back to the fractions $1/n$ with $2\le n<r-1$ (applying $L^{-1}$), to find the corresponding point $(1/n, r/n^2, +1)\in \WW$, and stop going up at $n_*$ if the right sister $n_*/(n_*-1)$ satisfies $\xi(n_*/(n_*-1))=n_*(n_*-1) \le r$. In this case we go to the right (applying $RL^{-1}=SR^{-1}$) to the fraction $n_*/(n_*-1)$, corresponding to the point $(n_*/(n_*-1), r/(n_*-1)^2, -1)\in \WW$ if $n_*(n_*-1) < r$, and to the point $(n_*/(n_*-1), n_*/(n_*-1), 0)\in \WW$ if $n_*(n_*-1) = r$;
\item[(f)] repeat the same algorithm: first go left until possible with points with $\eps=-1$ or $\eps=0$; then go back up with points with $\eps=+1$ until the energy of the right sister of the fraction is small enough to go right with $\eps=-1$; eventually we reach a fraction which is a right descendant (hence it has no right sisters) and in correspondence with a point with $\eps=+1$, from which we cannot go down; then we go up back visiting all the already seen fractions, this time with $\eps=+1$, until we get to the fraction $n_*/(n_*-1)$ of step (e), in correspondence with the point $(n_*/(n_*-1),r/(n_*-1)^2,+1)\in \WW$;
\item[(g)] move up to the fraction $1/(n_*-1)$, the direct parent of $1/n_*$ and $n_*/(n_*-1)$, (applying $R^{-1}$) obtaining the point $(1/(n_*-1),r/(n_*-1)^2,+1)$. The same pattern described in (e) and (f) repeats, that is we move to the right to the fraction $(n_*-1)/(n_*-2)$ with $\eps=-1$, and then apply step (f), until we reach again $(n_*-1)/(n_*-2)$, this time with $\eps=+1$, and then go back to $1/(n_*-2)$;
\item[(h)] repeat step (g) moving up until reaching $1/2$ in correspondence with the point $(1/2,r/4,+1)\in \WW$, then go right to $2/1$ with $\eps=-1$;
\item[(i)] now apply the steps in (b)-(g) by interchanging left with right. This follows from the symmetry of the Stern-Brocot trees, for which a reflection of the tree along a vertical line passing through $1/1$ sends the fraction $p/q$ to the fraction $q/p$. In particular the energy of these specular fractions is the same. Then we visit all the fractions with energy less than or equal to $r$ which lie on this right side of $\hat{\TTT}$, until we reach back $2/1$ in correspondence with the point $(2/1,r,+1)\in \WW$. The last step is to go back to $1/1$ and to $(1,r,+1)$, the last point in the periodic orbit of $(1,r,-1)$.
\end{enumerate}

\newpage
\begin{landscape}
    \begin{figure}[h]
        \begin{tikzpicture}[->,>=stealth',auto,node distance=3cm,
  thick,main node/.style={circle,draw=white,font=\sffamily\Large\bfseries}]

  \def\v{0.9cm}
  \def\h{0.045cm}
  \def\ff{40}
  \def\f{-22}
  \def\q{-50}
  \def\hs{0.15cm}

  \node[main node] (1-) {$(1,-1)$};
  \node[main node] (1+) [right = \hs of 1-] {$(1,+1)$};

  \node[main node] (2l+) [below left = \v and \ff*\h of 1-] {$(\frac 12, +1)$};
  \node[main node] (2r-) [below right= \v and \ff*\h of 1+] {$(\frac 21, -1)$};
  \node[main node] (2l-) [left = \hs of 2l+] {$(\frac 12, -1)$};
  \node[main node] (2r+) [right= \hs of 2r-] {$(\frac 21, -1)$};

  \node[main node] (3ll+) [below left = \v and \f*\h of 2l-] {$(\frac 13, +1)$};
  \node[main node] (3lr-) [below right= \v and \f*\h of 2l+] {$(\frac 32, -1)$};
  \node[main node] (3rl+) [below left = \v and \f*\h of 2r-] {$(\frac 23, +1)$};
  \node[main node] (3rr-) [below right= \v and \f*\h of 2r+] {$(\frac 31, -1)$};
  \node[main node] (3ll-) [left = \hs of 3ll+] {$(\frac 13, -1)$};
  \node[main node] (3lr+) [right= \hs of 3lr-] {$(\frac 32, +1)$};
  \node[main node] (3rl-) [left = \hs of 3rl+] {$(\frac 23, -1)$};
  \node[main node] (3rr+) [right= \hs of 3rr-] {$(\frac 31, +1)$};

  \node[main node] (4lll+) [below left = \v and \q*\h of 3ll-] {$(\frac 14, +1)$};
  \node[main node] (4lrr-) [below right= \v and \q*\h of 3lr+] {$(\frac 52, -1)$};
  \node[main node] (4rll+) [below left = \v and \q*\h of 3rl-] {$(\frac 25, +1)$};
  \node[main node] (4rrr-) [below right= \v and \q*\h of 3rr+] {$(\frac 41, -1)$};
  \node[main node] (4lll-) [left = \hs of 4lll+] {$(\frac 14, -1)$};
  \node[main node] (4lrr+) [right= \hs of 4lrr-] {$(\frac 52, +1)$};
  \node[main node] (4rll-) [left = \hs of 4rll+] {$(\frac 25, -1)$};
  \node[main node] (4rrr+) [right= \hs of 4rrr-] {$(\frac 41, +1)$};

  \node[main node] (5-) [below left = \v and \f*\h of 4lll-] {$\ \ \,\iddots\ \ \,$};
  \node[main node] (5+) [right= \hs of 5-] {$\ \ \,\iddots\ \ \,$};

  \node[main node] (6-) [below left = \v and \f*\h of 5-] {$(\frac 19, -1)$};
  \node[main node] (6+) [right= \hs of 6-] {$(\frac 19, +1)$};

  \node[main node] (7-) [below left = \v and \f*\h of 6-] {$(\frac 1{10}, -1)$};
  \node[main node] (7+) [right= \hs of 7-] {$(\frac 1{10}, +1)$};

  \node[main node] (5x+) [below right = \v and \f*\h of 4rrr+] {$\ \ \,\ddots\ \ \,$};
  \node[main node] (5x-) [left= \hs of 5x+] {$\ \ \,\ddots\ \ \,$};

  \node[main node] (6x+) [below right = \v and \f*\h of 5x+] {$(\frac 91, +1)$};
  \node[main node] (6x-) [left= \hs of 6x+] {$(\frac 91, -1)$};

  \node[main node] (7x+) [below right = \v and \f*\h of 6x+] {$(\frac {10}1, +1)$};
  \node[main node] (7x-) [left= \hs of 7x+] {$(\frac {10}1, -1)$};

  \path[every node/.style={font=\sffamily\small}]
  (1-)  edge[bend right] node {} (2l-)
  (2l-) edge[bend right] node {} (3ll-)
  (3ll-) edge[bend right] node {} (4lll-)
  (4lll-) edge[bend right] node {} (5-)
  (5-) edge[bend right] node {} (6-)
  (6-) edge[bend right] node {} (7-)
  (7-) edge[bend right] node {} (7+)
  (7+) edge[bend right] node {} (6+)
  (6+) edge[bend right] node {} (5+)
  (5+) edge[bend right] node {} (4lll+)
  (4lll+) edge[bend right] node {} (3ll+)
  (3ll+) edge node {} (3lr-)
  (3lr-) edge[bend right] node {} (4lrr-)
  (4lrr-) edge[bend right] node {} (4lrr+)
  (4lrr+) edge[bend right] node {} (3lr+)
  (3lr+) edge[bend right] node {} (2l+)
  (2l+) edge node {} (2r-)
  (2r-) edge[bend right] node {} (3rl-)
  (3rl-) edge[bend right] node {} (4rll-)
  (4rll-) edge node {} (4rll+)
  (4rll+) edge[bend right] node {} (3rl+)
  (3rl+) edge node {} (3rr-)
  (3rr-) edge[bend right] node {} (4rrr-)
  (4rrr-) edge[bend right] node {} (5x-)
  (5x-) edge[bend right] node {} (6x-)
  (6x-) edge[bend right] node {} (7x-)
  (7x-) edge[bend right] node {} (7x+)
  (7x+) edge[bend right] node {} (6x+)
  (6x+) edge[bend right] node {} (5x+)
  (5x+) edge[bend right] node {} (4rrr+)
  (4rrr+)edge[bend right] node {}  (3rr+)
  (3rr+) edge[bend right] node {} (2r+)
  (2r+) edge[bend right] node {} (1+)
  (1+) edge[bend right] node {} (1-)
  ;
\end{tikzpicture}
        \caption{Representation on the permuted Stern-Brocot tree of the
          periodic orbit of $(1,r,-1)$ for $\PPP_h$ in the case
          $r\in (10,11)$.}\label{fig-motions}
    \end{figure}
\end{landscape}


%

\end{document}